\documentclass[12pt]{article}
\usepackage{amsmath,amssymb,amsfonts}
\usepackage{amsthm}
\usepackage{enumerate}
\usepackage[dvipsnames]{xcolor}

\usepackage[top=1in, bottom=1in, left=1.0in, right=1.0in]{geometry}
\def\Xint#1{\mathchoice
{\XXint\displaystyle\textstyle{#1}}%
{\XXint\textstyle\scriptstyle{#1}}%
{\XXint\scriptstyle\scriptscriptstyle{#1}}%
{\XXint\scriptscriptstyle\scriptscriptstyle{#1}}%
\!\int}
\def\XXint#1#2#3{{\setbox0=\hbox{$#1{#2#3}{\int}$ }
\vcenter{\hbox{$#2#3$ }}\kern-.59\wd0}}
\def\avgint{\Xint-}

\newcommand{\vep}{\varepsilon}
\newcommand{\N}{\mathbb{N}}
\newcommand{\R}{\mathbb{R}}

\newcommand{\Q}{\mathbb{Q}}
\newcommand{\Om}{\Omega}

\newcommand{\Mod}{\textnormal{Mod}}

\newcommand{\diam}{\normalfont\text{diam}}
\newcommand{\rad}{\normalfont\text{rad}}

\DeclareMathOperator{\rcapa}{cap}

\newcommand{\mres}{\mathbin{\vrule height 2ex depth 2.2pt width
0.12ex\vrule height -0.3ex depth 2.2pt width .5ex}}
\newcommand{\mbdy}{\partial^*}
\newcommand{\abdy}{\partial^\alpha}
\newcommand{\fin}{f^{-1}}

\newtheorem{theorem}{Theorem}[section]
\newtheorem{proposition}[theorem]{Proposition}

\newtheorem{lemma}[theorem]{Lemma}

\theoremstyle{definition}
\newtheorem{defn}[theorem]{Definition}

\theoremstyle{remark}
\newtheorem{remark}[theorem]{Remark}

\numberwithin{equation}{section}

\begin{document}

\title{Modulus of families of sets of finite perimeter and  quasiconformal maps between metric spaces of globally $Q$-bounded
geometry}
\author{Rebekah Jones, Panu Lahti, Nageswari Shanmugalingam\footnote{The first and third authors' research
was partially supported by the grant DMS~\#1500440 of NSF (U.S.A.). Part of this research was conducted
during the visit of the three authors to Link\"oping University in Spring 2017 and Spring 2018; the authors wish to
thank that institution for its kind hospitality.}
}

\maketitle

\noindent{\small
{\bf Abstract} We generalize a result of Kelly~\cite{Kelly} to the setting of Ahlfors $Q$-regular metric measure spaces
supporting a $1$-Poincar\'e inequality. It is shown that if $X$ and $Y$ are two Ahlfors $Q$-regular spaces supporting
a $1$-Poincar\'e inequality and $f:X\to Y$ is a quasiconformal mapping, then the $Q/(Q-1)$-modulus of the collection
of measures $\mathcal{H}^{Q-1}\mres_{\Sigma E}$ corresponding to any collection of sets $E\subset X$ of finite perimeter
is quasi-preserved by $f$. We also show that for $Q/(Q-1)$-modulus almost every $\Sigma E$,
if the image surface $\Sigma f(E)$ does not see the singular set of $f$ 
as a large set, then $f(E)$ is also of finite perimeter. Even in the standard Euclidean 
setting our results are more general than that of Kelly, and hence are new even in there.
}

\bigskip


\medskip


\section{Introduction}

While classification of domains via conformal mappings gives a rich theory in the setting of planar domains, domains
in higher dimensional Euclidean spaces support no non-M\"obius conformal maps.
The most suitable geometric classification
in that setting is given by quasiconformal mappings. A homeomorphism $f:\Om\to\Om^\prime$ between two domains
$\Om,\Om^\prime\subset\R^n$ is quasiconformal if $f\in W^{1,n}_{loc}(\Om;\Om^\prime)$ and
there is a constant $K\ge 1$ such that whenever $x\in\Om$,
\[
\limsup_{r\to 0^+}\frac{\sup_{y\in \overline{B}(x,r)}|f(y)-f(x)|}{\inf_{y\in \Om\setminus B(x,r)}|f(y)-f(x)|}\le K.
\] 
The theory
of quasiconformal mappings was extended by Heinonen and Koskela in~\cite{HK} to the setting of metric measure spaces,
and in this non-smooth setting properties of quasiconformal mappings have been studied extensively, 
see for example~\cite{HK0, HK, HKST, Wil, BHW, KMS}. In this paper we continue this study by considering 
relationships between sets of finite perimeter and quasiconformal mappings in the spirit of~\cite{Kelly}.

The traditional perspective on quasiconformal mappings between Euclidean domains is that such a map is characterized
by its ability to quasi-preserve the conformal modulus of families of rectifiable curves in the respective domains. Thus
a homeomorphism $f:\Om\to\Om^\prime$ for two domains $\Om,\Om^\prime\subset\R^n$ is quasiconformal if there is 
a constant $C\ge 1$ such that whenever $\Gamma$ is a family of non-constant compact rectifiable curves in $\Om$,
\[
\frac{1}{C}\Mod_n(\Gamma)\le \Mod_n(f\Gamma)\le C\, \Mod_n(\Gamma).
\]
Here $f\Gamma$ is the family of curves obtained as images of curves in $\Gamma$ under $f$. An excellent discussion
about Euclidean quasiconformal mappings can be found in~\cite{Va}, see
also~\cite[Theorem~2.6.1]{Kelly}. A less well-known fact is that quasiconformal
mappings between two domains $\Om,\Om^\prime\subset\R^n$ quasi-preserve the
$\tfrac{n}{n-1}$-modulus of certain families of surfaces obtained as ``essential boundaries" of sets of finite perimeter.
This result is due to Kelly~\cite[Theorem~6.6]{Kelly}. 
In~\cite{Kelly} the families considered were the classes of sets $E\subset\Omega\subset\R^n$ of finite perimeter such that
$\mathcal{H}^{n-1}(\partial E)$ is finite  \emph{and} satisfies a double-sided
cone condition at \emph{every} point in $\partial E$,
see~\cite[Definition~6.1]{Kelly}. Kelly calls the boundary of such a set $E$ a \emph{surface}.
Building upon the Federer theory of differential forms for sets of finite perimeter
and the associated Gauss-Green theorem (see~\cite{Fed}), in~\cite{Kelly} it is shown that if $f$ is a quasiconformal
mapping, then for $\Mod_{n/(n-1)}$-almost every surface,
there is a change of variables formula, see~\cite[Theorem~4.7.1]{Kelly}.
Moreover, there is a family $\Sigma_0$ of
sets of finite perimeter in $\Om$ with $\Mod_{n/(n-1)}(\Sigma_0)=0$ such that whenever
$\partial E\subset \Om$ is a surface
with $E\not\in\Sigma_0$, then $f(E)$ is of finite perimeter in $\Om^\prime$ (\cite[Theorem~6.3]{Kelly}). 
However, there is a gap in the proof of~\cite[Theorem~6.3]{Kelly}, where the actual object studied is the reduced
boundary of $E$, denoted $\beta(E)$ in~\cite[page~372]{Kelly}, and this part of the boundary could be strictly smaller than the
measure-theoretic boundary of $E$.
According to Federer's characterization,
a Euclicean set is of finite perimeter if and only if the
$(n-1)$-dimensional Hausdorff measure of the measure-theoretic boundary of $E$ is finite. In~\cite{Kelly} it is shown
that for almost every $E$, $\beta (f(E))$
has finite $(n-1)$-dimensional Hausdorff measure, and this is not sufficient
to conclude that $f(E)$
is of finite perimeter.

The link between quasiconformal mappings and families of surfaces is natural also in light of the link
between quasiconformal mappings and moduli of families of curves, for in the Euclidean setting it is known
that there is a natural reciprocal connection between
families of surfaces separating two compacta and families of curves connecting the two compacta, see~\cite{AO}.
This link was already portended in~\cite[Lemma~5]{AB} (in planar geometry, the separating ``surfaces" are also curves).
Motivated by the results in~\cite{Kelly, AO}, the goal of this paper is to
prove a result similar to that
of~\cite[Theorem~6.6]{Kelly} for quasiconformal mappings between two complete metric measure spaces
equipped with an Ahlfors $Q$-regular (with $Q>1$) measure and
supporting a $1$-Poincar\'e inequality in the sense of
Heinonen and Koskela~\cite{HK}, and indeed, we consider families of sets from the
collection of \emph{all} sets of finite perimeter without the additional geometric 
constraints considered in~\cite{Kelly} (see Theorem~\ref{thm:main-1}).
To do so, we use the tools developed in~\cite{HK, HKST} regarding first-order calculus
on non-smooth spaces and the theory of BV functions first constructed in~\cite{M}, 
together with the result from~\cite{KMS} that quasiconformal maps are characterized by 
quasi-preserving the measure density of measurable subsets of $\Om$. This latter result is itself a generalization
of the work of Gehring and Kelly~\cite{GK}. Unlike in the work of Kelly~\cite{Kelly}, 
the notions of the Gauss-Green theorem and differential forms are not
available in the metric setting, and instead, we adapt the geometric measure theory tools developed in~\cite{BHW} in the metric
setting to verify an analog of the change of variables formula for sets of finite perimeter in the non-smooth setting.
The results of~\cite{BHW} are not directly applicable to our setting as neither the measure-theoretic boundary $\mbdy E$
(see Definition~\ref{def:meas-thr-bdy})
nor the essential boundary $\Sigma E$ of a set $E$ of finite perimeter is a $(Q-1)$-set (that is, it is not
Ahlfors $(Q-1)$-regular).
Here $\Sigma E$ is the subset of the boundary of $E$
made up of points that see both $E$ and its complement as having positive \emph{lower density}, see
Definition~\ref{def:gamma} below.
Therefore in this paper we combine some of the techniques of~\cite{BHW} with
the current technology on sets of finite perimeter to conduct a careful analysis of the images of $\Sigma E$.

In what follows, both $X$ and $Y$ are complete metric spaces equipped with an Ahlfors $Q$-regular measure 
for some $Q>1$, $f:X\to Y$ a quasiconformal homeomorphism,
$\mathcal{L}$ denotes the collection of measures $\mathcal{H}^{Q-1}\mres_{\Sigma E}$ corresponding
to sets $E\subset X$ that are of finite perimeter in $X$, and $f\mathcal{L}$ is the corresponding collection of 
measures $\mathcal{H}^{Q-1}\mres_{\Sigma f(E)}$. 

The quantities 
$L_f$ and $l_f$ represent the following:
\begin{equation}\label{eq:dilation}
L_f(x)=\limsup_{r\to 0^+}\frac{\sup_{y\in \overline{B}(x,r)}d_Y(f(x),f(y))}{r},\ \ 
l_f(x)=\liminf_{r\to0^+}\frac{\inf_{y\in X\setminus B(x,r)}d_Y(f(x),f(y))}{r},
\end{equation}
with $d_Y$ denoting the metric on the space $Y$, see Definition~\ref{def:dilations} below.
Here, for $1\le p<\infty$,
\[
\Mod_p(\mathcal{L})=\inf\bigg\lbrace\int_X\rho^p\, d\mathcal{H}^Q\, :\, \rho\text{ non-negative Borel with }
\int_{\Sigma E}\rho\, d\mathcal{H}^{Q-1}\ge 1
 \text{ for each }E\in\mathcal{L}\bigg\rbrace,
\]
see Definition~\ref{def:modulus} below. 
Here $\mathcal{L}$ has a dual identity, one as a collection of sets  $E\subset X$ of finite perimeter, and
the other as the collection of measures $\mathcal{H}^{Q-1}\mres_{\Sigma E}$.
In considering the quasiconformal images $f\mathcal{L}$, the family
$f\mathcal{L}$ stands for both the collection $f(E)$, $E\in\mathcal{L}$, and also for
the measures $\mathcal{H}^{Q-1}\mres_{\Sigma f(E)}$.

The following is the main result of this paper.

\begin{theorem}\label{thm:main-1}
Let $X,Y$ be two complete
Ahlfors $Q$-regular metric spaces, $Q>1$, that support a $1$-Poincar\'e inequality, 
and let $f:X\to Y$ be a quasiconformal map.
Then there exists $C>0$ such that for every collection $\mathcal{L}$ of
bounded sets of  finite perimeter measure in $X$ we have that 
\begin{equation}\label{eq:conserve-mod}
\Mod_{Q/(Q-1)}(\mathcal{L})\le C\, \Mod_{Q/(Q-1)}(f\mathcal{L}) 
\end{equation}
and
\begin{equation}\label{eq:conserve-mod-inverse}
\Mod_{Q/(Q-1)}(f\mathcal{L}^{\prime})\le C\,\Mod_{Q/(Q-1)}(\mathcal{L}^{\prime})
\end{equation}
where $f\mathcal{L}$ denotes the collection of images under $f$ and
$\mathcal{L}^{\prime}$ consists
of all $E\in\mathcal{L}$ for which
$0<L_{f}(x)<\infty$ for $\mathcal{H}^{Q-1}$-almost every $x\in \Sigma E$.
\end{theorem}

The above theorem gives new results even in the Euclidean setting, addressing the wider class 
of \emph{all} sets of finite perimeter rather than just those that satisfy a cone property at each point of the topological boundary,
with the topological boundary of finite Hausdorff $(n-1)$-dimensional measure as considered in~\cite[Definition~6.1]{Kelly}.

In proving Theorem~\ref{thm:main-1} we also show that for $\Mod_{Q/(Q-1)}$-almost every set $E\subset X$ of finite perimeter
the pull-back measure under $f$ of $\mathcal{H}^{Q-1}\mres_{\Sigma f(E)\cup[\mbdy f(E)\setminus f(P)]}$
is absolutely continuous with respect to $\mathcal{H}^{Q-1}\mres_{\Sigma E}$, with its Radon-Nikodym derivative estimated
by $J_f^{(Q-1)/Q}$, see Lemma~\ref{thm:J_E^(Q/Q-1)-leq-Jac} and Proposition~\ref{prop:abs-cont}.
Here
\[
P=\{x\in X :\, L_f(x)=\infty\text{ or }L_f(x)=0\}
\]
is the singular set of $f$.
We also address the question of whether images of sets of finite perimeter are of finite perimeter. There are examples
of planar quasiconformal mappings that map the unit disk to the von Koch snowflake domain,
and so map a set of 
finite perimeter to a set that is not of finite perimeter.
Hence we cannot expect
images of all sets of finite perimeter to be of finite perimeter, see also~\cite{BHW}.
From the results in~\cite{HK,HKST} we know that $\mathcal{H}^Q(P)=0$. Recall from there also that
both $f$ and $f^{-1}$ satisfy Lusin's condition $N$, that is, for sets 
$K\subset X$, $\mathcal{H}^Q(f(K))=0$ if and only if $\mathcal{H}^Q(K)=0$,
or equivalently,
$f_\#\mathcal{H}^Q_Y\ll\mathcal{H}^Q_X\ll f_\#\mathcal{H}^Q_Y$ (see~\cite[Theorem~8.12]{HKST}). 
Following \cite{HKST}, the Radon-Nikodym derivative
of $f_\#\mathcal{H}^Q_Y$ with respect to $\mathcal{H}^Q_X$ is denoted by $J_f$. The quantities
$J_f$, $L_f^Q$, and $l_f^Q$ are comparable to each other almost everywhere in $X$, see Lemma~\ref{lem:Jac-geq-L^Q}
below.

\begin{theorem}\label{thm:main-3}
Let $X,Y$ be complete
Ahlfors $Q$-regular metric spaces, $Q>1$, that support a $1$-Poincar\'e inequality, 
and let $f:X\to Y$ be a quasiconformal map.
Let $\mathcal{L}_P$ be the collection of all sets $E\subset X$ of finite perimeter  with
$\mathcal{H}^{Q-1}(\mbdy E\cap P)>0$ or 
$\mathcal{H}^{Q-1}(f((\mbdy E\setminus \Sigma E)\cap P))>0$.
Then for $\Mod_{Q/(Q-1)}$-almost every
bounded set $E\subset X$ of positive and finite perimeter in $X$ that does not belong to $\mathcal{L}_P$, the set $f(E)$ is of 
finite perimeter in $Y$ and the pull-back measure satisfies
\begin{equation}\label{eq:mutual-abs-cont-of-surfaces}
f_\#\mathcal{H}^{Q-1}\mres_{\partial^*f(E)}\ll\mathcal{H}^{Q-1}\mres_{\partial^*E}\ll f_\#\mathcal{H}^{Q-1}\mres_{\partial^*f(E)}
\end{equation}
with Radon-Nikodym derivative
$J_{f,E}\cong J_f^{(Q-1)/Q}$ where the comparison constant
$C$ depends only on the Ahlfors regularity constants and the Poincar\'e inequality
constants of $X$ and $Y$, and on the quasiconformality constant of $f$. Furthermore,
\begin{equation}\label{eq:main mod estimates}
\frac{1}{C}\, \Mod_{Q/(Q-1)}(\mathcal{L}\setminus\mathcal{L}_P)
\le \Mod_{Q/(Q-1)}(f(\mathcal{L}\setminus\mathcal{L}_P))
\le C\, \Mod_{Q/(Q-1)}(\mathcal{L}\setminus\mathcal{L}_P).
\end{equation}
\end{theorem}

The above results are not as discordant with the theory of quasiconformal mappings as it might seem. The characterization
of quasiconformal mappings by quasi-preservation of $Q$-modulus of families of curves is too strong; one only needs
the quasi-preservation of $Q$-modulus of families of rectifiable curves that connect pairs of disjoint compact sets $K, F$. 
Such classes of curves are associated with the relative capacity $\rcapa_Q(K,F)$, and the super-level sets of the 
potential associated with this capacity give sets of finite perimeter whose perimeter sets separate $K$ from $F$, and
it is the families of such sets that connect quasiconformal maps to the BV theory.

We now consider the implications of the above Theorem~\ref{thm:main-3} for the Euclidean setting $\R^n$.
An analog of the above theorem found in~\cite{Kelly} is Theorem~6.3, but
the proof of~\cite[Theorem~6.3]{Kelly}
has a gap as explained above. However, in light of the restrictions placed on the sets $E$ considered in~\cite{Kelly},
we know that those sets $E$ satisfy $\partial E=\Sigma E$, and for such $E$ it follows from
Proposition~\ref{prop:image-alpha-bdy}{\rm(3)} that except for a $\Mod_{n/(n-1)}$-null family, 
we know that $\mathcal{H}^{n-1}(f(P\cap \partial E))=0$, and therefore
\cite[Theorem~6.3]{Kelly} follows from Theorem~\ref{thm:main-3}, that is, the gap in the proof found in~\cite{Kelly} is
filled by the above theorem.

The structure of the paper is as follows. In Section~2 we give the notations and definitions related
to function spaces and measure-theoretic aspects of sets used in this paper. In Section~3 we give a brief background
related to quasiconformal mappings between metric measure spaces, and in Section~4 we list the needed background
results related to the concepts described in the previous two sections. Here we also give proofs and/or references
to papers where the interested reader can find proofs of these results. We give a proof of Theorem~\ref{thm:main-1} in
Section~5, and the last section deals with the proof of Theorem~\ref{thm:main-3}. 
In the last section we also show that
there are large families of sets of finite perimeter whose images under quasiconformal mappings are also of finite perimeter,
and thus lie outside the collection $\mathcal{L}_P$
of Theorem~\ref{thm:main-3}.

\section{Notations and definitions}

In this section we gather together the basic definitions we need in this paper. The definitions used here
are extensions to the non-smooth setting of the natural notions in Euclidean setting discussed in the introduction.
In this section, $(X,d,\mu)$ is a complete metric measure space with $\mu$ a Radon measure.

Given $x\in X$ and $r>0$, we denote an open ball by
$B(x,r)=\{y\in X:\,d(y,x)<r\}$. 
Given that in a metric space a ball, as a set, could have more than one
radius and more than one
center, we will consider a ball to be also equipped with a radius and center;
thus two different balls might
correspond to the same set. We then denote
$\rad (B):=r$ as the pre-assigned radius of the ball $B$, and $aB:=B(x,ar)$.
If $X$ is connected (as it must be in order to support a Poincar\'e inequality),
and if $X\setminus B$ is non-empty, then
$\rad(B)\le \diam(B)\le 2\, \rad(B)$.

\begin{defn}
Let $A\subset X$. Then for $d\ge 0$, the \emph{$d$-dimensional Hausdorff measure} of $A$ is given by 
\[
\mathcal{H}^d(A)=\lim_{r\to 0^+}\inf\left\{\sum_{k\in I}\rad(B_k)^d\,\middle\vert\,A\subset\bigcup_{k\in I}B_k
\text{ where }\rad(B_k)\le r\text{ and } I\subset\N\right\}.
\]
\end{defn}

\begin{defn}
We say that
$(X,d,\mu)$ is \emph{Ahlfors $Q$-regular} if $X$ has at least two points and there is
a constant $C_A\ge 1$ 
such that whenever $x\in X$ and $0<r<2\diam(X)$, we have
\[
\frac{r^Q}{C_A}\le \mu(B(x,r))\le C_A\, r^Q.
\]
As a consequence, we get
\[
\frac{\mu(A)}{C_A}\le \mathcal{H}^Q(A)\le C_A\, \mu(A).
\]
\end{defn}

Given an open set $U\subset X$,
we write $u\in L^1_{loc}(U)$ if $u\in L^1(V)$ for every open
$V\Subset U$; this expression means that $\overline{V}$ is a compact
subset of $U$.
Other local spaces are defined analogously.

A \emph{curve} is a continuous mapping from an interval into
$X$, and a \emph{rectifiable} curve is a curve with finite length.

\begin{defn}\label{def:N1p}
Let $Y$ be a metric space with metric $d_Y$.
Given a function $u:X\to Y$, a Borel function $g_u:X\to[0,\infty]$ is said to be an \emph{upper gradient of $u$} 
if for every compact rectifiable curve $\gamma$
\[
d_Y(u(x),u(y))\le\int_\gamma g_u\ ds
\]
where $x$ and $y$ are the endpoints of $\gamma$.
Let $1\le p<\infty$.
A function $f:X\to Y$ is said to be in $N^{1,p}_{loc}(X;Y)$
if $f\in L^p_{loc}(X;Y)$ and there is an upper gradient $g$  of $f$ such that $g\in L^p_{loc}(X)$. If $Y=\R$ and
$f,g\in L^p(X)$ then we say that $f\in N^{1,p}(X)$.
\end{defn}

We refer the reader to~\cite{HKST, HKST15} for the details regarding mappings in $N^{1,p}_{loc}(X;Y)$.

\begin{defn}\label{def:modulus}
Let $\mathcal{M}$ be a collection of
measures on $X$. Then the admissible class of $\mathcal{M}$, 
denoted $\mathcal{A}(\mathcal{M})$, is the set of all positive Borel functions $\rho:X \to [0,\infty]$ such that
\[ 
\int_X \rho \ d\lambda \ge 1 
\]
for all $\lambda \in \mathcal{M}$. Then the \emph{$p$-modulus} of the family $\mathcal M$ is given by
\[
\Mod_p(\mathcal M)=\inf_{\rho \in \mathcal{A}(\mathcal M)} \int_X \rho^p \ d\mu. 
\]
\end{defn}

$\Mod_p$ is an outer measure on the class of all measures, see~\cite{F}.
There are two types of collections of measures associated with quasiconformal maps. Given
a collection $\Gamma$ of curves in $X$, we set $\Gamma$ to also denote the  
arc length measures restricted to each curve in $\Gamma$; for this collection of measures, the above notion of
$\Mod_p(\Gamma)$ agrees with the standard notion of the $p$-modulus of the family $\Gamma$ of curves 
from~\cite{Va, HK, HKST15}.
For a collection $\mathcal L$ of sets of finite perimeter in $X$, 
we consider the measure $\mathcal{H}^{Q-1}\mres_{\Sigma E}$
for each $E\in\mathcal L$; it is known that
this measure is comparable to the 
perimeter measure associated with $E$ as in~Definition~\ref{def:finite-perimeter},
see Theorem \ref{thm:Ambr1}.

\begin{defn}
The \emph{relative $p$-capacity} of two sets $E,F \subset X$ is given by 
\[
\rcapa_p(E,F) = \inf \int_X g_u^p \ d\mu
\]
where the infimum is over all upper gradients $g_u$ of all  
functions $u\in N^{1,Q}_{loc}(X)$ such that $u|_E\le 0$ and $u|_F\ge 1$. 
\end{defn}

\begin{defn}
We say that the space $X$ supports a \emph{$p$-Poincar\'e inequality} 
if there exist constants $C_P>0$ and $\lambda\ge 1$ 
such that for all open balls $B$ in $X$, all measurable functions $u$ on $\lambda B$ and 
all upper gradients $g_u$ of $u$,
\[
\avgint_{B}|u-u_B| \,d\mu \le C_P\,\rad(B)\left(\avgint_{\lambda B}g_u^p \,d\mu\right)^{1/p}.
\]
Here we denote the integral average of $u$ over $B$ by
\[
u_B:=\avgint_{B}u\,d\mu:=\frac{1}{\mu(B)}\int_B u\,d\mu.
\]
\end{defn}

One of the consequences of a space being complete and Ahlfors
regular and supporting a Poincar\'e inequality is that such a
metric space must necessarily be quasiconvex, that is, there is some constant $C_q\ge 1$ such that for every
$x,y\in X$ there is a rectifiable curve $\gamma$ with end points $x,y$ and length $\ell(\gamma)\le C_q\, d(x,y)$,
see~\cite[Proposition 4.4]{HaK}
or~\cite[Theorem 4.32]{BB}.
Thus a bi-Lipschitz change in the metric
results in $X$ being a geodesic space, that is, a 
quasiconvex space with the quasiconvexity constant $C_q=1$. Notions such as Poincar\'e inequality, quasiconformality,
upper gradients and functions of bounded variation (see below), and Hausdorff measure are quasi-invariant 
under a bi-Lipschitz change in the metric, hence we do not lose generality by assuming that $X$ is a
geodesic space. Geodesic spaces that support a Poincar\'e inequality do so even with $\lambda=1$, 
see~\cite{HaK} or~\cite[Theorem 4.39]{BB}.

\begin{defn}
For a measurable set $E \subset X$ and $x \in X$, we define the \emph{upper density of $E$} at $x$ by 
\[
\overline{D}(E,x)= \limsup_{r\to 0^+} \frac{\mu (B(x,r)\cap E)}{\mu (B(x,r))}
\]
and the \emph{lower density of $E$} at $x$ by 
\[
\underline{D}(E,x)=\liminf_{r\to 0^+} \frac{\mu(B(x,r)\cap E)}{\mu(B(x,r))}.
\]
\end{defn}

\begin{defn}\label{def:meas-thr-bdy}
For a set $E$, the \emph{measure-theoretic boundary} is the set
\[
\mbdy E=\{x\in X\, :\, \overline D(E,x)>0 \text{ and } \overline D(X\setminus E,x)>0\}.
\]

\end{defn}

\begin{defn}
For $u \in L^1_\text{loc} (X)$, the total variation of $u$ on an open set $U \subset X$ is given by
\[
\Vert Du\Vert (U)=\inf\left\{\liminf_{n\to\infty}\int_Ug_{u_n}\ d\mu\,\middle|\,
		(u_n)_{n\in\N}\subset\text{Lip}_\text{loc}(U),u_n\to u\text{ in } L^1_\text{loc}(U)\right\}.
\]
In the above, $g_{u_n}$ stands for an upper gradient of $u_n$ in $U$ (here we consider $U$ to be the metric measure
space with metric and measure inherited from $X$).
We say $u$ is of \emph{bounded variation} on $X$ (denoted $u\in BV(X)$) if $\Vert Du\Vert (X)<\infty$.  We say that
$u\in BV_{loc}(X)$ if $u\in BV(U)$ for each open set $U\Subset X$.
\end{defn}

It is shown in \cite{M} that $\Vert Du\Vert $ is a Radon measure for any $u\in BV_{loc}(X)$. We call 
$\Vert Du\Vert $ the variation measure of $u$. 

\begin{defn}\label{def:finite-perimeter}
A measurable set $E\subset X$ has \emph{finite perimeter} if $\chi_E$ is of bounded variation on $X$. We call 
$\Vert D\chi_E\Vert$ the \emph{perimeter measure} of $E$ and we will denote it $P(E,\cdot)$.
\end{defn}

\begin{defn}
We say that $X$ supports a \emph{relative isoperimetric inequality} if there exist constants $C_I>0$ and 
$\lambda\ge 1$ such that for all balls $B(x,r)$ and for all
measurable sets $E$, we have 
\[
\min\{\mu(B(x,r)\cap E),\mu(B(x,r)\setminus E)\}\le C_I r P(E,B(x,\lambda r)).
\]
\end{defn}

Again, with $X$ a geodesic space, we can choose $\lambda=1$. We know that if $X$ is Ahlfors regular and supports a
$1$-Poincar\'e inequality, then it supports a relative isoperimetric inequality, see for example~\cite[Theorem 4.3]{A1}.

\begin{defn}\label{def:gamma}
For $\beta>0$, let 
\[
\Sigma_\beta E =\left\{x\in\mbdy E \ \big|\ \underline{D}(E,x)\ge\beta \text{ and } \underline{D}(X\setminus E,x)\ge\beta\right\},
\]
and set 
\[
\Sigma E = \bigcup_{\beta\in (0,1)} \Sigma_\beta E.
\]
\end{defn}

See~\cite{A1} or Theorem~\ref{thm:Ambr1} below for connections between
$\Sigma_\beta E$, $\mbdy E$,  and the
perimeter measure $P(E,\cdot)$.

\noindent {\bf Standing assumptions on the metric spaces:} \label{StandAssuption}
Throughout this paper we will assume that both
$(X,d_X,\mu_X)$ and $(Y,d_Y,\mu_Y)$ are complete metric spaces that are
Ahlfors $Q$-regular for some $Q>1$ and support a $1$-Poincar\'e inequality.
We will also, without loss of
generality, assume that $X$ and $Y$ are geodesic spaces.
We will use the letter $C$ to denote various
constants that depend, unless otherwise specified,
only on the Ahlfors regularity constants and the Poincar\'e inequality constants of $X$, and
the value of $C$ could differ at each occurrence.

\section{Quasiconformal mappings}

In this section we gather together definitions related to the notion of quasiconformal 
mappings between two metric spaces.
Here, $(X,d_X,\mu_Y)$ and $(Y,d_Y,\mu_Y)$ are complete metric measure spaces with
$\mu_X$, $\mu_Y$ Radon measures
and $Q>1$ such that $\mu_X\approx\mathcal{H}^Q=\mathcal{H}^Q_X$ and 
$\mu_Y\approx\mathcal{H}^Q=\mathcal{H}^Q_Y$
as in the standing assumptions from Section~2. 
Recall also the definitions of $L_f$ and $l_f$ from~\eqref{eq:dilation}:

\begin{defn}\label{def:dilations}
Define $L_f:X \to \R$ by 
\[
L_f(x)=\limsup_{r\to 0}\frac{L_f(x,r)}{r}\ \text{ where }\ L_f(x,r)=\sup_{y\in \overline{B}(x,r)}d_Y(f(x),f(y)).
\]
Similarly, define $l_f:X\to\R$ by
\[
l_f(x)=\liminf_{r\to 0}\frac{l_f (x,r)}{r}\ \text{ where }\ l_f (x,r)=\inf_{y\in X\setminus B(x,r)}d_Y(f(x),f(y)).
\]
\end{defn} 

When $f$ is a homeomorphism, we always have
$l_f(x,r)\le L_f(x,r)$. When $f$ is a quasiconformal homeomorphism,
there is a constant $K_D$ such that $L_f(x,r)\le K_D\, l_f(x,r)$,
see for example~\cite{HK} or~\eqref{eq:infinitesimal-to-finite-dilatation}
below.

There are different geometric notions of quasiconformal maps 
on metric spaces. 

\begin{defn}\label{def:metric-qc}
The homeomorphism $f$ is \emph{metric quasiconformal} if
there is a constant $K_D\ge 1$ such that for all $x\in X$ we have
\[
  \limsup_{r\to 0^+}\frac{L_f(x,r)}{l_f(x,r)}\le K_D.
\]
When we need to emphasize the constant $K_D$ we say that $f$ is $K_D$-quasiconformal.

The map $f$ is \emph{geometric quasiconformal} if
there is a constant $K\ge 1$ such that whenever $\Gamma$ is a family
of non-constant compact rectifiable curves in $X$, we have
\[
\frac{1}{K}\Mod_Q(f\Gamma)\le \Mod_Q(\Gamma)\le K\, \Mod_Q(f\Gamma).
\]
\end{defn}

\begin{defn}
A homeomorphism $f:X\to Y$ is \emph{quasisymmetric} if there is  a homeomorphism $\eta:[0,\infty)\to[0,\infty)$ such that
for every distinct triple of points $x,y,z\in X$ and $t>0$, 
\[
\frac{d_X(x,y)}{d_X(x,z)}\le t\implies\frac{d_Y(f(x),f(y))}{d_Y(f(x),f(z))}\le\eta(t).
\]
\end{defn}

The notions of quasisymmetry, metric quasiconformality, and geometric quasiconformality are connected, 
see~\cite{HK, Wil} and Theorem~\ref{thm:qc-equivalences} below.  

\begin{defn}
A homeomorphism $f:X\to Y$ between Ahlfors $Q$-regular metric spaces satisfies \emph{Lusin's condition $(N)$} if whenever 
$A\subset X$ is such that $\mathcal{H}^Q(A)=0$ then $\mathcal{H}^Q(f(A))=0$. We say that $f$ satisfies condition 
$(N^{-1})$ if its inverse satisfies condition $(N)$.
\end{defn}

Quasiconformal maps satisfy both Lusin's condition $(N)$ and $(N^{-1})$, see Theorem~\ref{thm:qc-equivalences} below.

\begin{defn}
If $\nu_X$ is a Radon measure on $X$ and $\nu_Y$ is a Radon measure on $Y$ and $f:X\to Y$ is a homeomorphism, then 
the \emph{pull-back of the measure $\nu_Y$} is the measure on $X$ given by
\[
f_\#\nu_Y(D):=\nu_Y(f(D))
\]
whenever $D$ is a Borel subset of $X$.
Note that since $f$ is a homeomorphism, $f_\#\nu_Y$ defines a Borel measure.
\end{defn}

\begin{defn}
We define the (generalized) \emph{Jacobian of $f$} at the point $x\in X$ as follows: 
\[
J_f(x)= \limsup_{r\to 0^+} \frac{\mathcal{H}^Q(f(B(x,r)))}{\mathcal{H}^Q(B(x,r))}.
\]
\end{defn}
Note that $J_f$ is the Radon-Nikodym derivative of the pull-back measure
$f_\#\mathcal{H}^Q_Y$ with respect to 
$\mathcal{H}^Q_X$
and $\mathcal{H}^Q$ is a doubling measure, and so the limit supremum in the definition of $J_f$ is actually a limit at
	$\mathcal{H}^Q$-almost every $x$.

\begin{defn}\label{def:J_f,E}
For a set $E\subset X$ of finite perimeter, 
should $f_\#\mathcal{H}^{Q-1}\mres_{\Sigma f(E)}\ll \mathcal{H}^{Q-1}\mres_{\Sigma E}$, 
we define the \emph{$(Q-1)$-Jacobian of $f$} with respect to $\Sigma E$
by
\[
 J_{f,E}(x)=\lim_{r\to 0^+}\frac{\mathcal{H}^{Q-1}\mres_{\Sigma f(E)}(f(B(x,r)))}{\mathcal{H}^{Q-1}\mres_{\Sigma E}(B(x,r))}.
\]
\end{defn}

Given that $\mathcal{H}^{Q-1}(\mbdy E\setminus\Sigma E)=0$ (see Theorem~\ref{thm:Ambr1}), we can equivalently consider $J_{f,E}$ to be the
Radon-Nikodym derivative on $\mbdy E$.\\

\noindent {\bf Further standing assumptions:} In this paper, in addition to the standing assumptions listed at the 
end of Section~\ref{StandAssuption}, we will also
assume that $f:X\to Y$ is a quasiconformal mapping.

\section{Background results}

 In this section we will gather together some of the background 
results needed in the paper.

\begin{theorem}[{\cite[Theorem~5.3, Theorem~5.4]{A1}, \cite[Theorem~4.6]{AMP}}]\label{thm:Ambr1}
There exists $\gamma>0$ (depending only on the Ahlfors regularity constant of $\mu_X$ 
and the Poincar\'e inequality constants)
such that for any set of finite perimeter $E$, the 
perimeter measure $P(E,\cdot)$ is concentrated on $\Sigma_\gamma E$. Furthermore, 
$\mathcal{H}^{Q-1}(\mbdy E \setminus \Sigma_\gamma E)=0$ and there exist constants $\tilde{\alpha}>0$ 
and $C>0$ (again depending only on 
$C_A,C_P$, and $\lambda$) and a Borel function $\Theta_E:X\to [\tilde{\alpha},C]$ 
such that 
\[
P(E,B(x,r))=\int_{B(x,r)\cap\mbdy E}\Theta_E\ d\mathcal{H}^{Q-1}
		=\int_{B(x,r)\cap\Sigma_\gamma E}\Theta_E\ d\mathcal{H}^{Q-1},
\]
for any $x\in X$ and $r>0$. Consequently we have that 
\begin{equation}\label{eq:H and perimeter}
\tilde{\alpha}\,\mathcal{H}^{Q-1}(B(x,r)\cap\mbdy E)\le P(E,B(x,r))\le C\,\mathcal{H}^{Q-1}(B(x,r)\cap\mbdy E)
\end{equation}
and
\[
\mathcal{H}^{Q-1}(\mbdy E\setminus\Sigma E)=0.
\]
\end{theorem}

The results of~\cite{A1} did not need the measure to be Ahlfors regular, only that the measure be doubling, but
as Ahlfors regularity is a stronger condition, the results of~\cite{A1} hold here as well.

\begin{theorem}[{\cite[Theorem 1.1]{L}}]\label{thm:H^(Q-1)-meas-finite} 
If $E\subset X$ is measurable and $\mathcal{H}^{Q-1}(\mbdy E)<\infty$ then $E$ is of finite perimeter.
\end{theorem}

Now we turn to preliminary results related to quasiconformal mappings needed in the paper.

\begin{theorem}[{\cite[Theorem 9.8]{HKST}}]\label{thm:qc-equivalences}
Let $f:X\to Y$ be a homeomorphism between metric spaces of locally $Q$-bounded geometry. 
Then the following conditions are quantitatively equivalent:
\begin{enumerate}
\item $f$ is $H$-quasiconformal
for some $H\ge 1$,
\item There is a homeomorphism $\eta$ such that $f$ is
locally $\eta$-quasisymmetric,
\item $f\in N^{1,Q}_{loc}(X:Y)$ and $L_f(x)^Q\le KJ_f(x)$ for almost every $x\in X$
and a constant $K>0$,
\item There is some $L>0$ such that for all curve families $\Gamma$ in $X$,
\[
L^{-1}\Mod_Q(\Gamma)\le \Mod_Q (f\Gamma) \le L\,\Mod_Q(\Gamma).
\]
\end{enumerate}
Furthermore, if any one of these conditions holds, then both $f$ and $f^{-1}$ are quasiconformal and 
satisfy Lusin's condition (N).
\end{theorem}

\begin{remark}
A metric space is said to be of locally $Q$-bounded geometry if $X$ is separable, path connected, locally compact, 
locally uniformly Ahlfors $Q$-regular and satisfies the
\emph{Loewner condition} locally uniformly. Under our assumptions, $X$ is 
locally (even globally) of $Q$-bounded geometry.
\end{remark}

In fact, under our assumptions the quasiconformal mapping $f$ is necessarily
quasisymmetric.

\begin{proposition}\label{prop:f is quasisymmetric}
The quasiconformal mapping $f$ is quasisymmetric.
\end{proposition}
\begin{proof}
Whenever $A\subset X$ is bounded, $\overline{A}$ is compact, and then
since $f$ is continuous, $f(\overline{A})$ is compact and thus bounded, that is,
$f$ maps bounded sets to bounded sets.
Since $f:X\to Y$ is a homeomorphism, $f^{-1}$ also maps bounded sets to bounded sets,
and so $X$ is bounded if and only if $Y$ is bounded.
Now if $X$ is unbounded,
by \cite[Corollary 4.8, Theorem 5.7]{HK} we know that $f$ is quasisymmetric;
if $X$ is bounded, we know this
by \cite[Theorem 4.9]{HK}.
\end{proof}

The next lemma follows also from Theorem~\ref{thm:qc-equivalences}(3) 
together with the chain rule
applied to $f$ and to $f^{-1}$, but as the proof is simple we provide it here for the convenience of the reader.

\begin{lemma}\label{lem:Jac-geq-L^Q}
If $\eta$ is the homeomorphism of quasisymmetry for $f$, then
there exists $C>0$, depending only on the Ahlfors regularity constant of the measures on $X$ and $Y$ and
on $\eta(1)$, such that at every $x \in X$,
\[
\frac{L_f(x)^Q}{C} \le J_f(x)\le C\, L_f(x)^Q.
\] 
\end{lemma}

\begin{proof}
Fix $x\in X$, and choose $r>0$ small enough so that $X\setminus \overline{B}(x,r)$ is non-empty.
Let $w\in B(x,r)$ and $z\in X\setminus B(x,r)$. Then 
\[
\frac{d_Y(f(x),f(w))}{d_Y(f(x),f(z))} 
\le\eta\left(\frac{d_X(x,w)}{d_X(x,z)}\right) 
\le\eta(1) 
\] 
because $d(x,w)\le d(x,z)$. This holds for all $w\in B(x,r)$ and $z\in X\setminus B(x,r)$, and so 
\begin{equation}\label{eq:infinitesimal-to-finite-dilatation}
\frac{L_f(x,r)}{l_f(x,r)}\le\eta(1).
\end{equation}
Now for $y\in Y\setminus f(B(x,r))$ we have $d_X(\fin(y),x)\ge r$, and so 
\[
d_Y(y,f(x))\ge l_f(x,r)\ge \frac{L_f(x,r)}{\eta(1)}. 
\]
It follows that
$y \notin B(f(x),L_f(x,r)/\eta(1))$. Hence,
$B(f(x),L_f(x,r)/\eta(1))\subset f(B(x,r))$.
Now at
every $x\in X$,
	$J_f(x)$ is given by
\begin{align*}
J_f(x)=\limsup_{r\to 0^+}\frac{\mathcal{H}^Q(f(B(x,r)))}{\mathcal{H}^Q(B(x,r))}
	&\ge\limsup_{r\to 0^+}\frac{\mathcal{H}^Q\left(B\left(f(x),\frac{L_f(x,r)}{\eta(1)}\right)\right)}{\mathcal{H}^Q(B(x,r))} \\
	&\ge\limsup_{r\to 0^+}\frac{\frac 1{C_A}\left(\frac{L_f(x,r)}{\eta(1)}\right)^Q}{C_Ar^Q} \\ 
	&=\limsup_{r\to 0^+}\frac{1}{C_A^2 \,\eta(1)^Q} \left(\frac{L_f(x,r)}{r}\right)^Q\\
	&=\frac{1}{C_A^2\,\eta(1)^Q} {L_f(x)}^Q,
\end{align*}
and
\[
J_f(x)\le\limsup_{r\to 0^+}\frac{\mathcal{H}^Q(B(f(x),L_f(x,r)))}{\mathcal{H}^Q(B(x,r))} 
	\le C_A^2\, L_f(x)^Q.
\]
Letting $C=C_A^2\,\max\{\eta(1)^Q,1\}$, the conclusion follows.
\end{proof}

\begin{lemma}[{\cite[Lemma 2.7]{BHW}}]\label{lem:qs-inclusions}
For every ball $B(f(x),s)\subset Y$
there exists $r>0$ such that
\[
B(f(x),s)\subset f(B(x,r))\subset f(B(x,10r))\subset B(f(x),\eta(10)s).
\]
Furthermore, if $\fin$ is uniformly continuous with modulus of continuity $\omega(\cdot)$, then we can choose $r\le\omega(s)$.
\end{lemma}

\begin{proof}
Since $Y$ is proper and $\fin$ is continuous, there exists $r>0$ such that the
first inclusion holds. Let
$r=\inf\{r'\,|\,f(B(x,r'))\supset B(f(x),s)\}$;
since $f$ is a homeomorphism, $B(f(x),s)\subset f(B(x,r))$.
Then for any $0<c<1$, there exists a point $z_c\in \fin(B(f(x),s))\setminus B(x,cr)$. Then 
$f(z_c)\in B(f(x),s)\setminus f(B(x,cr))$ which implies that 
$d_Y(f(x),f(z_c)\le s$ and $d_X(x,z_c)\ge cr$. Let $w\in B(x,10r)$.
Now $f$ is quasisymmetric by Proposition \ref{prop:f is quasisymmetric},
with an associated homeomorphism $\eta$, and so
\[
d_Y(f(x),f(w)) \le \eta\left(\frac{d_X(x,w)}{d_X(x,z_c)}\right) d_Y(f(x),f(z_c))\le \eta\left(\frac {10}{c}\right)s.
\]
Letting $c$ tend to 1, we get that $d_Y(f(x),f(w))\le \eta(10)s$. Thus the last inclusion holds.
Now if $\omega(t)$ is a modulus of continuity of $\fin$, then $r\le\omega(s)$ since we chose $r$ minimally.
\end{proof}

Recall that $\Sigma E=\bigcup_{\beta\in(0,1)}\Sigma_\beta E$, see Definition~\ref{def:gamma}.

\begin{lemma}\label{lem:shift-density}
Let $E\subset X$ be measurable.
For each $\beta\in(0,1)$ there exists $\beta_0\in(0,1)$
such that $\Sigma_{\beta}f(E)\subset f(\Sigma_{\beta_0} E)$.
Consequently, $\Sigma f(E)= f(\Sigma E)$.
Also, $\mbdy f(E)=f(\mbdy E)$.
\end{lemma}

\begin{proof}
Let $x\in X$. By \cite[Theorem 6.2]{KMS} we know that for all sufficiently
small balls $B_1$ centered at $x$, we have for some $a,b> 0$
\begin{equation}\label{eq:uniform density property}
\frac{\mu_Y(f(E)\cap B_2)}{\mu_Y(B_2)}\le b\left(\frac{\mu_X(E\cap B_1)}{\mu_X(B_1)}\right)^a,
\end{equation}
where $B_2$ denotes the largest open ball in $f(B_1)$ with center $f(x)$.

Suppose $\beta\in(0,1)$ and $f(x)\in \Sigma_{\beta}f(E)$, that is, both
$\underline{D}(f(E),f(x))$ and $\underline{D}(Y\setminus f(E),f(x))$ are at least as large as $\beta$.
As the radius of $B_1$ converges to 0, so does the radius of $B_2$, and so
it follows from \eqref{eq:uniform density property} that
\[
\beta_0:=\left(\frac{\beta}{b}\right)^{1/a}\le \underline{D}(E,x).
\]
By using the fact that $\underline{D}(Y\setminus f(E),f(x))\ge\beta$,
and \eqref{eq:uniform density property} with $E$ replaced by $X\setminus E$, we
get also $\beta_0\le \underline{D}(X\setminus E,x)$.
Thus $x\in \Sigma_{\beta_0}E$, and so we have proved that
$\Sigma_{\beta}f(E)\subset f(\Sigma_{\beta_0} E)$.
It follows that $\Sigma f(E)\subset f(\Sigma E)$.
Since $f^{-1}$ is also quasiconformal, we also get
\[
\Sigma f^{-1}(f(E))\subset f^{-1}(\Sigma f(E))
\]
and so $f(\Sigma E)\subset \Sigma f(E)$.
We therefore have $f(\Sigma E)= \Sigma f(E)$.

Next suppose that $f(x)\in \partial^*f(E)$.
It follows from \eqref{eq:uniform density property} that
\[
0<\left(\frac{\overline{D}(f(E),x)}{b}\right)^{1/a}\le \overline{D}(E,x).
\]
From the fact that $\overline{D}(Y\setminus f(E),f(x))>0$,
and \eqref{eq:uniform density property} with $E$ replaced by $X\setminus E$, we
get also $0< \overline{D}(X\setminus E,x)$.
Thus $x\in \partial^*E$, and so we have proved that
$\partial^*f(E)\subset f(\partial^*E)$.
Since $f^{-1}$ is also quasiconformal, we also get
\[
\partial^* f^{-1}(f(E))\subset f^{-1}(\partial^* f(E))
\]
and so $f(\partial^* E)\subset \partial^* f(E)$, whence
we conclude that $f(\partial^* E)= \partial^* f(E)$.
\end{proof}

\begin{lemma}\label{lem:alpha-bdy}
There exists $\alpha>0$ such that for every $E\subset X$ of finite 
perimeter and for $\mathcal{H}^{Q-1}$-almost every $x \in \mbdy E$, 
\[
\liminf_{r\to 0} \frac{P(E,B(x,r))}{r^{Q-1}}>\alpha.
\]
Moreover,
if $0<\beta<1$, then
there is some $\alpha(\beta)>0$ such that whenever $E\subset X$ is a measurable set
and $x\in\Sigma_\beta E$, we have 
\[
\liminf_{r\to 0}\frac{P(E,B(x,r))}{r^{Q-1}}>\alpha(\beta).
\]
\end{lemma}

\begin{proof}
By Ahlfors $Q$-regularity we know that $r^{-1}\mu(B(x,r))$ is comparable to $r^{Q-1}$. 
Recall that $C_I$ is the constant from the isoperimetric inequality and $C_A$ is the Ahlfors regularity constant.
By Theorem~\ref{thm:Ambr1}, there exists $\gamma>0$ such that for $\mathcal{H}^{Q-1}$-almost every $x\in\mbdy E$,
\[
\gamma \le \liminf_{r\to 0}\frac{\mu(B(x,r)\cap E)}{\mu(B(x,r))} 
\quad\text{ and }\quad\gamma \le \liminf_{r\to 0}\frac{\mu(B(x,r)\setminus E)}{\mu(B(x,r))}.
\]

Then by the relative isoperimetric inequality,
\begin{align*}
\gamma \le & \liminf_{r\to 0^+} \frac{\min\{\mu(B(x,r) \cap E),\mu(B(x,r)\setminus E)\}}{\mu(B(x,r))} \\ 
	\le & \liminf_{r\to 0^+} \frac{r C_I P(E,B(x,r))}{\mu(B(x,r))} \\ 
	\le & \liminf_{r \to 0^+} \frac{C_I C_A P(E,B(x,r))}{r^{Q-1}}.
\end{align*}
Letting $\alpha=\frac{\gamma}{C_I C_A}$ concludes
the proof of the first part of the lemma.  
The second part of the lemma is proved in the same way as the  first part, with
\[
\alpha(\beta):=\frac{\beta}{C_I C_A}. \qedhere
\]
\end{proof}

\begin{defn}\label{def:alpha-bdy}
For $\alpha>0$, define
\[
\abdy E :=\Big\{x\in\mbdy E\Big|\liminf_{r\to 0}\frac{P(E,B(x,r))}{r^{Q-1}}>\alpha\Big\}.
\]
\end{defn}

\begin{remark}
By Lemma~\ref{lem:alpha-bdy} above, we have that for each $0<\beta<1$,
\[
\Sigma_\beta E\subset\partial^{\alpha(\beta)}E.
\]
\end{remark}

We need the following ``continuity from below" for families of measures, 
with the families not necessarily measurable with respect to the outer measure $\Mod_p$.

\begin{lemma}[{Ziemer's lemma~\cite[Lemma~3.1(3)]{BHW}}]\label{lem:mod-continuity}
Let $\{\mathcal{L}_i\}_{i\in\N}$ be a sequence of families of measures in $X$ such that for each $i$, 
$\mathcal{L}_i\subset \mathcal{L}_{i+1}$.
Then for $1<p<\infty$,
\[
\Mod_p\Big(\bigcup_{i\in\N}\mathcal{L}_i\Big) = \lim_{i\to\infty}\Mod_p(\mathcal{L}_i).
\]
\end{lemma}

\section{Proof of Theorem~\ref{thm:main-1}}

In this section we prove Theorem~\ref{thm:main-1}. To do so, we adapt the 
tools given in~\cite{BHW} to study boundaries of sets of finite perimeter, which are
\emph{not} Ahlfors regular in general. The adaptation of the tools to this setting is
given in Proposition~\ref{prop:image-alpha-bdy}.

Recall that $P=\{x\in X \,:\, L_f(x)=\infty \text{ or } L_f(x)=0\}$.
Note from Lemma~\ref{lem:Jac-geq-L^Q} that $J_f(x)=0$
for $\mathcal H^{Q}$-almost every $x$ for which $L_f(x)=0$. So
by the fact that $f$ and $f^{-1}$ satisfy condition $(N)$
(from Theorem~\ref{thm:qc-equivalences}), we must have that
the zero set of $J_f$ and hence the zero set of $L_f$ is a null set.
As $L_f\in L^1_{loc}(X)$, we know that 
$\mathcal{H}^Q(P)=0$.
We remind the reader of the standard assumptions set forth at the
end of Sections~2 and~3.
Recall also the definition of $\abdy E$ from Definition~\ref{def:alpha-bdy}.
As noted in the introduction, with a slight abuse in notation, 
we will use $\mathcal{L}$ to denote both the family of sets $E$ of
finite perimeter and the family of measures $\mathcal{H}^{Q-1}\mres_{\Sigma E}$.

If $\mathcal{L}$ contains a set $E$ with $P(E,X)=0$, i.e.~$\mathcal{H}^{Q-1}(\mbdy E)=0$,
then $\Mod_{Q/(Q-1)}(\mathcal{L})=\infty$ as there can be no admissible test function $\rho$ for the class
$\mathcal{L}$.   
In addition, by the fact that $X$ supports a $1$-Poincar\'e inequality we will also have that $\mathcal{H}^Q(E)=0$
or $\mathcal{H}^Q(X\setminus E)=0$. It then follows that $\mathcal{H}^Q(f(E))=0$
or $\mathcal{H}^Q(f(X\setminus E))=\mathcal{H}^Q(Y\setminus f(E))=0$, whence it follows that $f(E)$ is of 
finite perimeter with $P(f(E),Y)=0$. We conclude that $\Mod_{Q/(Q-1)}(f\mathcal{L})=\infty$.
In this case Theorem~\ref{thm:main-1} trivially holds true. So in the proof of the next proposition (and in the next section)
we will assume that every $E\in\mathcal{L}$ satisfies $0<P(E,X)<\infty$.

\begin{proposition}\label{prop:image-alpha-bdy}
Let $\mathcal{L}$ denote the given collection
of bounded sets $E\subset X$ of positive and finite perimeter measure. 
Then the following hold true:
\begin{enumerate}
\item for each $\alpha>0$ and $\Mod_{Q/(Q-1)}$-almost every $E\in\mathcal{L}$
we have $\mathcal{H}^{Q-1}(f(\abdy E))<\infty$,
\item if \ $U_0\subset X$ with $\mathcal{H}^{Q} (f(U_0)) = 0$ and $\alpha>0$, then we have that 
   $\mathcal{H}^{Q-1}(f(\abdy E \cap U_0))=0$ for
   $\Mod_{Q/(Q-1)}$-almost every $E\in\mathcal{L}$,
\item if \ $U_0\subset X$ with $\mathcal{H}^{Q} (f(U_0)) = 0$, then 
with $\mathcal{L}_{bad}$ the collection of all $E\in\mathcal{L}$ with 
$\mathcal{H}^{Q-1}(\Sigma f(E) \cap f(U_0))>0$, we have
$\Mod_{Q/(Q-1)}(\mathcal{L}_{bad})=0=\Mod_{Q/(Q-1)}(f\mathcal{L}_{bad})$.
\end{enumerate}
\end{proposition}

\begin{proof}
Recall from Lemma~\ref{lem:shift-density}
that $f(\mbdy E)=\mbdy f(E)$ and $f(\Sigma E)=\Sigma f(E)$.

Let $U\subset X$ be a bounded measurable set. 
Then since $f$ is a homeomorphism and closed and bounded subsets of $X$ are compact,
$f(U)$ is also bounded and
	$f$ and $f^{-1}$ are uniformly continuous
on the sets $U$ and $f(U)$, respectively.
Let $\omega(\cdot)$ be a modulus of continuity for
$f^{-1}$ on $f(U)$.
Let $\vep>0$. By the definition of Hausdorff measure, there exist $y_i\in f(U)$ and $0<s_i<\vep$ 
such that $\{B(y_i,s_i)\}_{i\in\N}$ covers $f(U)$ and $\sum_{i\in\N}s_i^{Q}\le (\mathcal{H}^{Q}(f(U))+\vep)$. 
By Lemma~\ref{lem:qs-inclusions}, 
for every $x_i:=f^{-1}(y_i)$, there exists $0<r_i\le\omega(\vep)$ such that 
\[
B(y_i,s_i)\subset f(B(x_i,r_i))\subset f(B(x_i,10r_i))\subset B(y_i,\eta(10)s_i).
\]
Set $B_i=B(x_i,r_i)$ and $B_i'=B(y_i,s_i)$ and define $g:X\to\R$ by 
\[
g(x)=\sup_{i\in\N} \left(\frac{s_i}{r_i}\right)^{Q-1}\chi_{2B_i}(x).
\]
Fix $\alpha>0$ and define 
\[
\abdy_\delta E=\left\{x\in\mbdy E \ \middle| \ \frac{P(E,B(x,r))}{r^{Q-1}} > \alpha \text{ for all } 0<r\le\delta \right\}
\] 
and note that $\abdy E = \bigcup_{k\in\N} \abdy_{1/k} E $ . For each $M>0$, define 
\[
\mathcal{L}^M_{U,\vep} =
  \left\{E\in\mathcal{L}\,\middle|\,\mathcal{H}^{Q-1}_{\vep \eta(10)}(f(\abdy_{\omega(\vep)}E\cap U))> M\right\}.
\]

We want to show that $C\eta(10)^{Q-1}[\alpha M]^{-1}\, g$ is admissible for calculating 
$\Mod_{Q/(Q-1)}(\mathcal{L}^M_{U,\vep})$. 
Let $E \in \mathcal{L}^M_{U,\vep}$ and set
$I_E = \{i\in\N \, \big| \, B_i\cap(\abdy_{\omega(\vep)} E \cap U) \neq \emptyset \}$. 
Note that $\abdy_{\omega(\vep)} E \cap U \subset \bigcup_{i\in I_E} 2B_i$.
Then by the $5$-Covering Lemma, there exists $J_E \subset I_E$ such that 
$\{2B_j\}_{j\in J_E}$ is pairwise disjoint and 
$\abdy_{\omega(\vep)} E \cap U \subset \bigcup_{j\in J_E} 10B_j$. 
Since $B_j \cap \abdy_{\omega(\vep)} E \cap U \neq \emptyset$ for each $j\in J_E$, 
there exists $z_j \in B_j \cap \abdy_{\omega(\vep)} E \cap U$. 
So $B(z_j,r_j) \subset 2B_j$. Moreover, $z_j \in \abdy_{\omega(\vep)} E$ and 
$r_j \le \omega(\vep)$ imply that $r_j^{-(Q-1)} P(E,B(z_j,r_j)) > \alpha$, and hence
$P(E,2B_j) \ge P(E,B(z_j,r_j)) > \alpha r_j^{Q-1}$. Then by the pairwise
disjointness property of $\{2B_j\}_{j\in J_E}$,
\begin{align*} 
\int_{\Sigma E} g \ d\mathcal{H}^{Q-1} 
        \ge \int_{\Sigma E} \Big(\sup_{j\in J_E} \left(\frac{s_j}{r_j}\right)^{Q-1} \chi_{2B_j} \Big) \ d\mathcal{H}^{Q-1} 
	& = \int_{\Sigma E} \Big(\sum_{j\in J_E} \left(\frac{s_j}{r_j}\right)^{Q-1} \chi_{2B_j} \Big) \ d\mathcal{H}^{Q-1} \\
	& \ge \frac{1}{C}\sum_{j\in J_E} \left(\frac{s_j}{r_j}\right)^{Q-1} P(E,2B_j)\quad
	\textrm{by }\eqref{eq:H and perimeter}\\
	& \ge \frac{1}{C}\sum_{j\in J_E} \left(\frac{s_j}{r_j}\right)^{Q-1} \alpha \, r_j^{Q-1}  \\
	& = \frac{\alpha}{C} \sum_{j\in J_E} s_j^{Q-1}. \\
\end{align*}
Thus we have that 
\begin{equation}\label{eq:B1}
\int_{\Sigma E} g \ d\mathcal{H}^{Q-1} \ge \frac{\alpha}{C} \sum_{j\in J_E} s_j^{Q-1}.
\end{equation} 
Because  $\displaystyle U\cap\abdy_{\omega(\vep)} E \subset \bigcup_{j\in J_E} 10B_j$, 
\[
f(\abdy_{\omega(\vep)} E \cap U) \subset f \Big(\bigcup_{j\in J_E} 10B_j\Big) \subset \bigcup_{j\in J_E} B(y_j, \eta(10)s_j).
\]
Since $E\in \mathcal{L}^M_{U, \vep}$, we have that $\mathcal{H}^{Q-1}_{\vep\eta(10)} (f(\abdy_{\omega(\vep)}E \cap U)) > M$.
As $s_j<\vep$, $\bigcup_{j\in J_E} \eta(10) B_j'$ is an admissible cover for computing 
$\mathcal{H}^{Q-1}_{\vep\eta(10)}(f(\abdy_{\omega(\vep)}E \cap U))$,
and hence 
\[
  M < \mathcal{H}^{Q-1}_{\vep\eta(10)} (f(\abdy_{\omega(\vep)}E  \cap U)) 
	\le \sum_{j\in J_E} (\eta(10)s_j)^{Q-1} 
	= \eta(10)^{Q-1} \sum_{j\in J_E} s_j^{Q-1}.
\]
So 
\begin{equation}\label{eq:B2}
\frac{M}{\eta(10)^{Q-1}} <  \sum_{j\in J_E} s_j^{Q-1}.
\end{equation}
Combining \eqref{eq:B1} and \eqref{eq:B2}, we get 
\[
\int_{\Sigma E} g \ d\mathcal{H}^{Q-1} \ge \frac{\alpha}{C} \left(\frac{M}{\eta(10)^{Q-1}}\right).
\]
Therefore 
\[
\frac{C\, \eta(10)^{Q-1}}{\alpha\, M} g
\]
is admissible for computing $\Mod_{Q/(Q-1)}(\mathcal{L}_{U,\vep}^M)$.
Setting
\[
C_M=\frac{C\, \eta(10)^{Q-1}}{\alpha\, M},
\]
we obtain 
\begin{align*}
\Mod_{Q/(Q-1)} (\mathcal{L}^M_{U,\vep})
  &\le C_A \int_X (C_M g)^{Q/(Q-1)} \ d\mathcal{H}^Q \\
  &= C_M^{Q/(Q-1)} C_A \int_X \sup_{i\in \N} \left(\left(\frac{s_i}{r_i}\right)^{Q-1} \chi_{2B_i} \right)^{Q/(Q-1)} d\mathcal{H}^Q \\
  &\le C_M^{Q/(Q-1)} C_A \int_X \left(\sum_{i\in\N} \left(\frac{s_i}{r_i}\right)^{Q} \chi_{2B_i}\right) d\mathcal{H}^Q \\
  &\le C_M^{Q/(Q-1)} C_A \sum_{i\in\N} \left(\frac{s_i}{r_i}\right)^{Q} \mathcal{H}^Q(2B_i) \\
  &\le C_M^{Q/(Q-1)} C_A \sum_{i\in\N} \left(\frac{s_i}{r_i}\right)^{Q} C_A\, (2r_i)^Q \\ 
  &= C_M^{Q/(Q-1)} C_A^2 2^Q \sum_{i\in\N} s_i^Q \\
  &\le C_M^{Q/(Q-1)} C_A^2 2^Q (\mathcal{H}^Q(f(U)) + \vep). \\ 
\end{align*}

Let $\mathcal{L}^M_U = \bigcup_{k\in\N} \mathcal{L}^M_{U,\, 1/k}$. 
Then 
\begin{equation}\label{eq:LargeSet}
\mathcal{L}^M_U = \{E \in \mathcal{L} \,|\, \mathcal{H}^{Q-1}(f(\abdy E \cap U))>M\}.
\end{equation}
Note that if $m<k$ then $\mathcal{L}^M_{U,\, 1/m}\subset\mathcal{L}^M_{U,\, 1/k}$.
Applying Lemma~\ref{lem:mod-continuity}, we get that
\begin{align}\label{eq:A1}
\Mod_{Q/(Q-1)} (\mathcal{L}^M_U) =\lim_{k\to\infty} \Mod_{Q/(Q-1)} (\mathcal{L}^M_{U,1/k}) 
  &\le\lim_{k\to\infty}C_M^{Q/(Q-1)}C_A^2 2^Q\Big(\mathcal{H}^{Q}(f(U))+\frac 1k\Big) \notag\\
  &=C_M^{Q/(Q-1)}C_A^2 2^Q \mathcal{H}^{Q}(f(U)). 
\end{align}

\vskip .3cm

\noindent {\bf Proof of Claim~\textit{1}:}
Set $\mathcal{L}_U^{\infty} = \{ E \in \mathcal{L} \,|\, \mathcal{H}^{Q-1}(f(\abdy E \cap U))=\infty\}$.
Then for each $M>0$ we have $\mathcal{L}_U^M \supset \mathcal{L}_U^\infty$, and so
\begin{align*}
\Mod_{Q/(Q-1)}(\mathcal{L}_U^\infty) 
\le \lim_{M\to\infty}\Mod_{Q/(Q-1)}(\mathcal{L}^M_U)
& \le \lim_{M\to\infty} C_M^{Q/(Q-1)} C_A^2 2^Q \mathcal{H}^{Q}(f(U)) \\
& = \lim_{M\to\infty} \frac{1}{M^{Q/(Q-1)}} \left(\frac{C\eta(10)^Q}{\alpha^{Q/(Q-1)}}\right) C_A^2 2^Q\mathcal{H}^{Q}(f(U)) \\  
& = 0.
\end{align*}
Fix $x_0 \in X$ and considering $U_i = B(x_0,i)$ for each $i\in\N$, we set
$\mathcal{L}^\infty = \bigcup_{i\in\N} \mathcal{L}^\infty_{U_i}$. 
Then as the sets in $\mathcal{L}$ are bounded, 
$\mathcal{L}^\infty = \{ E \in \mathcal{L} \,|\, \mathcal{H}^{Q-1}(f(\abdy E))=\infty\}$
and 
\[
\Mod_{Q/(Q-1)} (\mathcal{L}^\infty) 
	= \Mod_{Q/(Q-1)} \Big(\bigcup_{i\in\N} \mathcal{L}^\infty_{U_i} \Big) 
	\le \sum_{i\in\N} \Mod_{Q/(Q-1)} (\mathcal{L}^\infty_{U_i}) 
	= 0. 
\]
Therefore the set for which $\mathcal{H}^{Q-1}(f(\abdy E))$ 
is not finite has $Q/(Q-1)$-modulus zero. This proves Claim~\textit{1}.

\vskip .3cm

\noindent {\bf Proof of Claim~\textit{2}:}  Let $U_0 \subset X$ such that $\mathcal{H}^{Q}(f(U_0)) = 0$. 
Then for any bounded $U \subset U_0$, $\mathcal{H}^{Q}(f(U))=0$. 
Recall $\mathcal{L}^M_U$ from~\eqref{eq:LargeSet} for
$M>0$, and 
let $\mathcal{L}^+_U=\bigcup_{m\in\N}\mathcal{L}^{1/m}_U$. 
Then 
\[
\mathcal{L}^+_U = \{ E \in \mathcal{L} \,|\, \mathcal{H}^{Q-1}(f(\abdy E \cap U))>0\},
\]
and by~\eqref{eq:A1} and by the outer measure property of modulus,
\begin{align*}
\Mod_{Q/(Q-1)}\left(\mathcal{L}^+_U\right) =\Mod_{Q/(Q-1)}(\bigcup_{m\in\N} \mathcal{L}^{1/m}_U) 
	&\le\sum_{m\in\N}\Mod_{Q/(Q-1)}(\mathcal{L}^{1/m}_U) \\
	&\le\sum_{m\in\N}C_{1/m}^{Q/(Q-1)}C_A^2 2^Q\mathcal{H}^{Q}(f(U)) \\
	&=0.
\end{align*}
Define $V_i=B(x_0,i)\cap U_0$ and set $\mathcal{L}^+ =\bigcup_{i\in\N}\mathcal{L}^+_{V_i}$. 
Then 
\[
\mathcal{L}^+=\{E\in\mathcal{L}\,|\,\mathcal{H}^{Q-1}(f(\abdy E\cap U_0))>0\}, 
\]
and 
\[
\Mod_{Q/(Q-1)}(\mathcal{L}^+)\le\sum_{i\in\N}\Mod_{Q/(Q-1)}(\mathcal{L}^+_{V_i})=0. 
\]
 
\vskip .3cm
 
\noindent {\bf Proof of Claim~\textit{3}:}
Again, suppose $\mathcal{H}^Q(f(U_0))=0$.
By Lemma~\ref{lem:alpha-bdy}, for any $0<\beta<1$ there exists $\alpha(\beta)>0$ such that 
$\Sigma_\beta E\subset\partial^{\alpha(\beta)}E$. 
Then by Claim~\textit{2}, for $\Mod_{Q/(Q-1)}$-almost every $E$
\begin{align*}
\mathcal{H}^{Q-1}(f(\Sigma E \cap U_0)) =\mathcal{H}^{Q-1}\left(\bigcup_{\beta\in (0,1)}f(\Sigma_\beta E\cap U_0)\right) 
	&\le\mathcal{H}^{Q-1}\left(\bigcup_{\beta\in (0,1)\cap\Q} f(\partial^{\alpha(\beta)}E\cap U_0)\right) \\
	&\le\sum_{\beta\in (0,1)\cap\Q}\mathcal{H}^{Q-1}(f(\partial^{\alpha(\beta)}E\cap U_0)) \\
	&=0.
\end{align*}
Finally, by considering the function $\infty\chi_{f(U_0)}$, we see that
$\Mod_{Q/(Q-1)}(f\mathcal{L}_{bad})=0$.
\end{proof}

\begin{proposition}\label{prop:abs-cont}
Let $\mathcal{L}$ denote the given collection of  bounded sets of finite perimeter in $X$.
For $\Mod_{Q/(Q-1)}$-almost every $E\in\mathcal{L}$ we have
\[
f_\#\mathcal{H}^{Q-1}\mres_{\Sigma f(E)}=
f_\#\mathcal{H}^{Q-1}\mres_{f((\mbdy E\setminus P)\cup\Sigma E)}
\ll \mathcal{H}^{Q-1}\mres_{\mbdy E} = \mathcal{H}^{Q-1}\mres_{\Sigma E}
\]
\end{proposition}

\begin{proof}
The last equality follows from Theorem~\ref{thm:Ambr1}.
To prove the absolute continuity, first
recall that $P=\{x\in X \,|\, L_f(x)\in\{0,\infty\} \}$. For $n,m\in\N$, set
\begin{equation}\label{eq:defn-Anm}
A_{n,m}=\bigg\{x\in X \,\bigg|\, \frac1m\le\frac{L_f(x,r)}{r} \le m \text{ for } 0<r<\frac1n\bigg\}.
\end{equation}
Then $\bigcup_{n\in\N} \bigcup_{m\in\N} A_{n,m} = X\setminus P$.
If $x,y\in A_{n,m}$ such that $d_X(x,y)<1/n$, then 
\begin{equation}\label{eq:Lip}
d_Y(f(x),f(y))\le m\, d_X(x,y),
\end{equation}
that is, $f$ is locally $m$-Lipschitz on $A_{n,m}$. Furthermore,
from~\eqref{eq:infinitesimal-to-finite-dilatation}
we have that $l_f(x,r)\ge L_f(x,r)/\eta(1)$, and so
\begin{equation}\label{eq:biLip}
d_Y(f(x),f(y))\ge \frac{1}{\eta(1) m}\, d_X(x,y),
\end{equation}
that is, $f$ is locally $\max\{1,\eta(1)\} m$-biLipschitz on $A_{n,m}$.
(In this proof, we only need $f$ to be Lipschitz, not biLipschitz.)
Note that $\mathcal{H}^Q(P)=0$
and thus $\mathcal{H}^Q(f(P))=0$,
and so by Proposition~\ref{prop:image-alpha-bdy}(3),
for $\Mod_{Q/(Q-1)}$-almost every $E\in\mathcal{L}$ we have 
$\mathcal{H}^{Q-1}(f(P\cap \Sigma E))=0$. For such $E\in\mathcal{L}$, 
fix $N\subset\mbdy E$ such that $\mathcal{H}^{Q-1}(N)=0$. 
Then for each $m,n\in\N$, by the local $m$-Lipschitz property of $f$ on $A_{n,m}$, we have 
\[
\mathcal{H}^{Q-1}\left(f\big(N\cap A_{n,m}\big)\right)
		\le m^{Q-1}\, \mathcal{H}^{Q-1}\left(N\cap A_{n,m}\right)=0.
\]
Thus
\[
\mathcal{H}^{Q-1}(f(N\setminus P))=\mathcal{H}^{Q-1}\bigg(f\big(N\cap \bigcup_{n,m\in\N}A_{n,m}\big)\bigg)
		\le \sum_{n,m\in\N}m^{Q-1}\, \mathcal{H}^{Q-1}\left(N\cap A_{n,m}\right)=0,
\]
 and therefore
\[
\mathcal{H}^{Q-1}(f(((\mbdy E\setminus P)\cup\Sigma E)\cap N)) 
\le \mathcal{H}^{Q-1}(f(N\setminus P)) +\mathcal{H}^{Q-1}(f(N\cap \Sigma E \cap P)) =0.
\]
Finally, note that $\mathcal{H}^{Q-1}(\partial^*E\setminus(P\cup\Sigma E))\le \mathcal{H}^{Q-1}(\partial^*E\setminus\Sigma E)=0$, 
and so the first identity in the proposition also holds by the absolute continuity of the pull-back measure
together with the fact that $\Sigma f(E)= f(\Sigma E)$ (see Lemma~\ref{lem:shift-density}).
\end{proof}

Recall the definition of $J_{f,E}$ from Definition~\ref{def:J_f,E}, and note that it is a function on $\Sigma E$.
 
\begin{lemma}\label{lem:finite-measure-of-Sigma}
For $\Mod_{Q/(Q-1)}$-almost every $E\in\mathcal{L}$ we have $\mathcal{H}^{Q-1}(\Sigma f(E))<\infty$.
\end{lemma}

\begin{proof}
By Lemma~\ref{lem:alpha-bdy} there exists $\alpha_0>0$ such that 
$\mathcal{H}^{Q-1}(\Sigma E \setminus \partial^{\alpha_0} E)=0$
for every $E\in\mathcal L$.
Then by Proposition~\ref{prop:abs-cont} we know that 
$\mathcal{H}^{Q-1}(\Sigma f(E)\setminus f(\partial^{\alpha_0} E))=0$ for $\Mod_{Q/(Q-1)}$-almost every $E\in\mathcal{L}$.
Finally, by Proposition~\ref{prop:image-alpha-bdy}(1) we have $\mathcal{H}^{Q-1}(f(\partial^{\alpha_0} E))<\infty$
after eliminating a further family of $\Mod_{Q/(Q-1)}$-zero from $\mathcal{L}$.
\end{proof}

\begin{lemma}\label{thm:J_E^(Q/Q-1)-leq-Jac}
There exists $C>0$ such that for $\Mod_{Q/(Q-1)}$-almost every $E\in\mathcal{L}$, 
\[
J_{f,E}(x) \le C {J_f(x)}^{(Q-1)/Q}
\]
for $\mathcal{H}^{Q-1}\mres_{\Sigma E}$-almost every $x$.
\end{lemma}

\begin{proof}
From Proposition~\ref{prop:abs-cont} we know that for $\Mod_{Q/(Q-1)}$-almost every $E$,
$f_\#\mathcal{H}^{Q-1}\mres_{\Sigma f(E)}$ is absolutely continuous with respect to 
$\mathcal{H}^{Q-1}\mres_{\mbdy E}=\mathcal{H}^{Q-1}\mres_{\Sigma E}$.
Furthermore, from Lemma~\ref{lem:finite-measure-of-Sigma} we know that $\mathcal{H}^{Q-1}(\Sigma f(E))<\infty$ 
for $\Mod_{Q/(Q-1)}$-almost every $E$. We focus only on such $E\in\mathcal{L}$ now. 

Note that $L_f$ is a locally integrable function on $X$, so by Lusin's Theorem, for each $k\in\mathbb{N}$ 
there is some open set $U_k\subset X$
such that $\mu(U_k)\le 2^{-k}$ and $L_f\vert_{X\setminus U_k}$ is continuous. 
By enlarging $U_k$ if necessary, we can also assume that for each $k\in\N$,
$P\subset U_k$.
Let $U=\bigcap_{k\in\N} U_k$. Then $\mu(U)=0$.
Let $\mathcal{L}_0=\{E\in\mathcal{L}:\mathcal{H}^{Q-1}(U\cap\Sigma E)>0\}$.
Then $\infty\chi_U$ is admissible for computing $\Mod_{Q/(Q-1)}(\mathcal{L}_0)$, and so
$\Mod_{Q/(Q-1)}(\mathcal{L}_0)\le \int_X \infty\chi_U \ d\mu=0$. We ignore such $E$ as well.

Observe that for every $x\in X\setminus U$ we have $L_f(x)>0$.
For $n\in\mathbb{N}$ we set
\[
A_{n}=\{x\in\Sigma E :\, L_f(x,r)/r\le 2L_f(x)\text{ for }0<r<1/n\}. 
\]
It is not difficult to show that
$A_{n}$ is a Borel set by writing it as an intersection of Borel sets, one for each rational $r\in(0,1/n)$. 
Now by Proposition~\ref{prop:abs-cont} and the Radon-Nikodym Theorem, 
\[
\int_{\Sigma E} J_{f,E} \ d\mathcal{H}^{Q-1} = \mathcal{H}^{Q-1}(f(\Sigma E)).
\]
Since $\mathcal{H}^{Q-1}(f(\Sigma E))=\mathcal{H}^{Q-1}(\Sigma f(E))$ is finite, 
$J_{f,E}\in L^1(\Sigma E, \mathcal{H}^{Q-1})$.
For $k,n\in\mathbb{N}$ let $E_{n}^k$ denote the collection of all $x\in\Sigma E$
that are not Lebesgue points of $\chi_{A_{n}\setminus U_k}J_{f,E}$. Then for each $k\in\mathbb{N}$ we have
$\mathcal{H}^{Q-1}(\bigcup_{n\in\mathbb{N}}E_{n}^k)=0$. For
$x\in \Sigma E\setminus \bigcap_{k\in\mathbb{N}}\left(U_k\cup\bigcup_{n\in\mathbb{N}}E_{n}^k\right)$,
there is some $k\in\mathbb{N}$ and $n\in\mathbb{N}$ such that $x\in A_{n}$ but 
$x\notin E_{n}^k\cup U_k$. Therefore
\begin{align}\label{eq:A2}
J_{f,E}(x) &=\lim_{r\to 0^+} \avgint_{B(x,r)\cap\Sigma E}\chi_{A_{n}\setminus U_k}\, J_{f,E} \ d\mathcal{H}^{Q-1} \notag \\
	&=\lim_{r\to 0^+}\frac{\int_{B(x,r)\cap A_{n}\setminus U_k}J_{f,E} \ d\mathcal{H}^{Q-1}}{\mathcal{H}^{Q-1}(B(x,r)\cap\Sigma E)} 
	\notag \\
	&=\lim_{r\to 0^+}\frac{f_{\#}\mathcal{H}^{Q-1}(B(x,r)\cap A_{n}\setminus U_k)}{\mathcal{H}^{Q-1}(B(x,r)\cap\Sigma E)}.
\end{align}
From an argument similar to~\eqref{eq:Lip} we know that $f$ is $2\sup_{B(x,r)\setminus U_k}L_f$--Lipschitz continuous
on $B(x,r)\cap A_n\setminus U_k$ when  $0<r<1/2n$, and so
\begin{align*}
f_{\#}\mathcal{H}^{Q-1}(B(x,r)\cap A_{n}\setminus U_k)
 &= \mathcal{H}^{Q-1}(f(B(x,r)\cap A_n\setminus U_k))\\
& \le \left[2\sup_{B(x,r)\setminus U_k}L_f\right]^{Q-1}\, \mathcal{H}^{Q-1}(B(x,r)\cap A_n\setminus U_k)\\
&\le \left[2\sup_{B(x,r)\setminus U_k}L_f\right]^{Q-1}\, \mathcal{H}^{Q-1}(B(x,r)\cap\Sigma E),
\end{align*}
and so by~\eqref{eq:A2} and the continuity of $L_f$ in $X\setminus U_k$,
\[
J_{f,E}(x)\le \limsup_{r\to 0^+}\left[2\sup_{B(x,r)\setminus U_k}L_f\right]^{Q-1}=2^{Q-1}\, L_f(x)^{Q-1}.
\]
Now by applying Lemma~\ref{lem:Jac-geq-L^Q}, we obtain
\[
J_{f,E}(x) \le C{J_f(x)}^{(Q-1)/Q}.
\]
This holds for all $x\in\Sigma E\setminus (\bigcap_{k\in\N} (U_k\cup \bigcup_{n\in\N}E_n^k))$. As
$E\notin \mathcal L_0$ and $\mathcal{H}^{Q-1}(\bigcap_{k\in\N}\bigcup_{n\in\N} E^k_n)=0$, we have
\[
\mathcal{H}^{Q-1}\left(\Sigma E\cap\bigcap_{k\in\N} \left(U_k\cup \bigcup_{n\in\N}
E_n^k\right)\right)=0,
\]
and so the claim holds.
\end{proof}

\begin{proof}[Proof of Theorem~\ref{thm:main-1}]  
Let $\mathcal{L}_0$ be the set of $E\in\mathcal{L}$ for which the conclusion of either Proposition~\ref{prop:abs-cont} or 
Lemma~\ref{thm:J_E^(Q/Q-1)-leq-Jac} fails.  Then by those results, we know that $\Mod_{Q/(Q-1)}(\mathcal{L}_0)=0$.
Let $E\in\mathcal L\setminus \mathcal{L}_0$.

From Proposition~\ref{prop:abs-cont} and Lemma~\ref{lem:shift-density}
we know that the pull-back measure
$f_\# \mathcal{H}^{Q-1}\mres_{\Sigma f(E)}=f_\# \mathcal{H}^{Q-1}\mres_{f(\Sigma E)}$
	is absolutely continuous with respect to
$\mathcal{H}^{Q-1}\mres_{\Sigma E}$. 
Therefore, whenever $\varphi$ is a nonnegative Borel function on $Y$, we have
\[
\int_{ f(\Sigma E)}\varphi\, d\mathcal{H}^{Q-1}=\int_{\Sigma E}\varphi\circ f\, J_{f,E}\, d\mathcal{H}^{Q-1}.
\]
From Lemma~\ref{thm:J_E^(Q/Q-1)-leq-Jac} we know that
the corresponding Radon-Nikodym derivative is dominated by $C_0\, J_f^{(Q-1)/Q}$.
Suppose $\rho:Y\to[0,\infty]$ is admissible for computing
$\Mod_{Q/(Q-1)} (f\mathcal{L})$. 
Define $\widetilde\rho:X\to[0,\infty]$ by 
\[
\widetilde\rho=C_0\, (\rho\circ f){J_f}^{(Q-1)/Q}.
\]
Then $\widetilde\rho$ is admissible for calculating $\Mod_{Q/(Q-1)}(\mathcal{L}\setminus\mathcal{L}_0)$  
since by Proposition~\ref{prop:abs-cont} and the change of variables formula,
\begin{align*}
 \int_{\Sigma E} \widetilde\rho \ d\mathcal{H}^{Q-1} 
	=  \int_{\Sigma E} C_0\,  (\rho\circ f) {J_f}^{(Q-1)/Q} \ d\mathcal{H}^{Q-1}  
	&\ge \int_{\Sigma E} (\rho\circ f) J_{f,E} \ d\mathcal{H}^{Q-1} \\
	&= \int_{f(\Sigma E)} \rho \ d\mathcal{H}^{Q-1} \\
	&= \int_{\Sigma f(E)}\rho\ d\mathcal{H}^{Q-1} \ge1.
\end{align*}
It follows that
\begin{align*}
\Mod_{Q/(Q-1)}(\mathcal{L})= \Mod_{Q/(Q-1)}(\mathcal{L}\setminus\mathcal{L}_0) 
	&\le C_A\int_X \widetilde\rho^{Q/(Q-1)} \ d\mathcal{H}^Q \\
	&= C_A\int_X \left(C_0\, (\rho\circ f){J_f}^{(Q-1)/Q}\right)^{Q/(Q-1)} \ d\mathcal{H}^Q  \\
	&\le C \int_X (\rho\circ f)^{Q/(Q-1)} J_f \ d\mathcal{H}^Q  \\
	&= C \int_Y \rho^{Q/(Q-1)} \ d\mathcal{H}^Q .
\end{align*}
Taking the infimum over admissible $\rho$, we obtain
\[
\frac 1C \Mod_{Q/(Q-1)}(\mathcal{L})\le \Mod_{Q/(Q-1)}(f\mathcal{L}).
\]
We cannot directly apply the argument to $\fin$ with respect to
$f\mathcal{L}^{\prime}$ as
we do not know that that family consists solely of sets of finite perimeter. For
$E\in\mathcal{L}$ we know that $f(E)\in f\mathcal{L}$ is of finite perimeter if 
and only if $\mathcal{H}^{Q-1}(\mbdy f(E))<\infty$, see Theorem~\ref{thm:H^(Q-1)-meas-finite}.
A priori we do not know whether $\mathcal{H}^{Q-1}(\mbdy f(E))=\infty$ implies that
$\mathcal{H}^{Q-1}(f(\Sigma E))=\infty$.
So instead we provide a direct proof.
Recall that $P=\{x\in X\, :\, L_f(x)=0\}\cup\{x\in X\, :\, L_f(x)=\infty\}$, and 
as in~\eqref{eq:defn-Anm}, we set
\[
A_{n,m}=\bigg\{x\in X \,\bigg|\, \frac1m\le\frac{L_f(x,r)}{r} \le m \text{ for } 0<r<\frac1n\bigg\}.
\]
Then, as noted in~\eqref{eq:biLip}, $f$ is $\max\{1,\eta(1)\} m$--biLipschitz on $A_{n,m}$ and
\[
X\setminus P=\bigcup_{n,m\in\N}A_{n,m}.
\]
By the definition of $\mathcal{L}^{\prime}$ we have that
$\mathcal{H}^{Q-1}(P\cap\Sigma E)=0$. Therefore
\[
\mathcal{H}^{Q-1}
	\left(\Sigma E\setminus \bigcup_{n,m\in\N} A_{n,m}\right)=0.
\]
As $f$ is locally (that is, at scales smaller than $1/n$) biLipschitz on each $A_{n,m}$, we now have that
\begin{equation}\label{eq:meas-abs-cont}
\mathcal{H}^{Q-1}\mres_{\Sigma E}\ll f_\#\mathcal{H}^{Q-1}\mres_{\Sigma f(E)}.
\end{equation}
Equivalently,
$f^{-1}_\#\mathcal{H}^{Q-1}\mres_{\Sigma E}\ll \mathcal{H}^{Q-1}\mres_{\Sigma f(E)}$.
Now a direct adaptation of the proof of Lemma~\ref{thm:J_E^(Q/Q-1)-leq-Jac} tells us that
the Radon-Nikodym derivative of $f^{-1}_\#\mathcal{H}^{Q-1}\mres_{\Sigma E}$ with respect to 
$\mathcal{H}^{Q-1}\mres_{\Sigma f(E)}$, denoted $J_{f^{-1},fE}$, satisfies
for $\Mod_{Q/(Q-1)}$-almost every $f(E)$,
with $E\in \mathcal{L}^{\prime}$, 
that
\[
J_{f^{-1},fE}(y)\le C\, J_{f^{-1}}(y)^{(Q-1)/Q},
\]
for $\mathcal{H}^{Q-1}\mres_{\Sigma f(E)}$-almost every $y$.
Now we can apply the proof of the first part of this theorem
to obtain that
if $\rho:X\to[0,\infty]$ is admissible for computing $\Mod_{Q/(Q-1)}(\mathcal{L}^{\prime})$,
then 
\[
\widehat{\rho}=C_1(\rho\circ f^{-1})\, J_{f^{-1}}^{(Q-1)/Q}
\]
is admissible for computing $\Mod_{Q/(Q-1)}(f\mathcal{L}^{\prime})$. Thus 
we get the desired comparison stated in the last part of the theorem.
\end{proof}

\section{Proof of Theorem~\ref{thm:main-3}}

In this section we focus on the proof of Theorem~\ref{thm:main-3}. The proof uses the tools developed in
the previous section. Recall that 
\[
\mathcal{L}_P=\{E\in\mathcal{L} : \mathcal{H}^{Q-1}(f((\mbdy E\setminus \Sigma E)\cap P))>0 \text{ or }
 \mathcal{H}^{Q-1}(P\cap\Sigma E)>0\}.
\]

\begin{proposition}\label{prop:f(E)-finite-perim}
For $\Mod_{Q/(Q-1)}$-almost every $E\in\mathcal{L}\setminus\mathcal{L}_P$, 
$\mathcal{H}^{Q-1}(\mbdy f(E))$ is finite and hence $f(E)$ is of finite perimeter.
\end{proposition}

\begin{proof} 
By Lemma~\ref{lem:alpha-bdy} there exists $\alpha>0$ such that $\mathcal{H}^{Q-1}(\mbdy E \setminus \abdy E)=0$.
By Proposition~\ref{prop:image-alpha-bdy}(1), for
$\Mod_{Q/(Q-1)}$-almost every $E$ we have that $\mathcal{H}^{Q-1}(f(\abdy E))$ is finite.
Then by Proposition~\ref{prop:abs-cont}, we have $\mathcal{H}^{Q-1}(f(\mbdy E\setminus (P\cup \abdy E)))=0$.
For $E\in\mathcal{L}\setminus\mathcal{L}_P$ that satisfy these conditions (and $\Mod_{Q/(Q-1)}$-almost all of them do),
\begin{align*}
\mathcal{H}^{Q-1}(f(\mbdy E))
&\le \mathcal{H}^{Q-1}(f(\abdy E))+\mathcal{H}^{Q-1}(f(\mbdy E\setminus (P\cup \abdy E)))\\
&= \mathcal{H}^{Q-1}(f(\abdy E))<\infty.
\end{align*}
An application of Theorem~\ref{thm:H^(Q-1)-meas-finite}
and Lemma~\ref{lem:shift-density} completes the proof.
\end{proof}

From the above chain of inequalities we see that for
$\Mod_{Q/(Q-1)}$-almost all $E\in\mathcal{L}\setminus\mathcal{L}_P$
we have $\mathcal{H}^{Q-1}(f(\mbdy E))=\mathcal{H}^{Q-1}(f(\abdy E))$
and thus
\begin{equation}\label{eq:mbdy vs Sigma E}
\mathcal{H}^{Q-1}(f(\mbdy E))=\mathcal{H}^{Q-1}(f(\Sigma E)).
\end{equation}

\begin{proof}[Proof of Theorem~\ref{thm:main-3}]
The above proposition implies the validity of the first claim in the theorem.

By Proposition~\ref{prop:abs-cont} and \eqref{eq:mbdy vs Sigma E},
 the first absolute continuity of \eqref{eq:mutual-abs-cont-of-surfaces}
 holds for $\Mod_{Q/(Q-1)}$--almost every $E\in\mathcal{L}\setminus\mathcal{L}_P$.
 By \eqref{eq:meas-abs-cont} and the fact that
 $\mathcal{H}^{Q-1}\mres_{\partial^*E}=\mathcal{H}^{Q-1}\mres_{\Sigma E}$,
 the second holds for
 $\Mod_{Q/(Q-1)}$--almost every
 $E\in\mathcal{L}^{\prime}\supset \mathcal{L}\setminus\mathcal{L}_P$.

The first inequality of~\eqref{eq:main mod estimates}
is verified by applying the first part of Theorem~\ref{thm:main-1} to 
$\mathcal{L}\setminus\mathcal{L}_P$ and $f$.
Next note that if $E\in \mathcal{L}\setminus \mathcal{L}_P$, then
$E\in \mathcal{L}^\prime$.
Thus the second inequality of~\eqref{eq:main mod estimates}
follows from the second part of Theorem~\ref{thm:main-1}.
To complete the proof, we note from Lemma~\ref{thm:J_E^(Q/Q-1)-leq-Jac}
and~\eqref{eq:mutual-abs-cont-of-surfaces}, 
\eqref{eq:main mod estimates}
that $J_{f,E}(x)\le C J_f^{(Q-1)/Q}(x)$
and $J_{f^{-1},f(E)}(f(x))\le C\, J_{f^{-1}}^{(Q-1)/Q}(f(x))$
for $\mathcal{H}^{Q-1}\mres_{\Sigma E}$-almost every $x$,
for $\Mod_{Q/(Q-1)}$-almost every $E\in\mathcal{L}\setminus\mathcal{L}_P$.
By the chain rule for Radon-Nikodym derivatives, we know that
$J_{f^{-1}}(f(x))=1/J_f(x)$
for $\mathcal{H}^{Q}$-almost every $x\in X$ and thus for
$\mathcal{H}^{Q-1}\mres_{\Sigma E}$-almost every $x$,
for $\Mod_{Q/(Q-1)}$-almost every $E\in\mathcal{L}\setminus\mathcal{L}_P$.
Similarly, $J_{f^{-1},fE}(f(x))=1/J_{f,E}(x)$.
Thus we have that
\[
J_{f,E}\le C J_f^{(Q-1)/Q}=\frac{C}{(J_{f^{-1}}\circ f)^{(Q-1)/Q}}
\le \frac{C^2}{J_{f^{-1},f(E)}\circ f}=C^2\, J_{f,E}.
\]
This completes the proof of the theorem.
\end{proof}

\begin{remark}\label{rem:MadayanPL}
We now complete the discussion in this paper by considering the reasonableness of excluding 
$\mathcal{L}_P$. The natural question
to ask here is: how much are we losing by neglecting $\mathcal{L}_P$?
Suppose that for all $E\in\mathcal{L}$ we have $f(E)$ is of finite perimeter in $Y$. Then 
$\mathcal{H}^{Q-1}([\mbdy f(E)\setminus\Sigma f(E)]\cap f(P))=0$.
Since $\mathcal{H}^Q(P)=0$, we have
	$\mathcal{H}^{Q-1}(P\cap\Sigma E)=0$ for
	$\Mod_{Q/(Q-1)}$-almost every $E\in \mathcal L$.
Note also that
as $\mathcal{H}^Q(P)=\mathcal{H}^Q(f^{-1}(f(P)))=0$, for $\Mod_{Q/(Q-1)}$-almost every $f(E)$ we have that 
$\mathcal{H}^{Q-1}(f^{-1}(f(P))\cap\Sigma E)=0$ by Proposition~\ref{prop:image-alpha-bdy}(3)
applied to the family $f(E)$, $E\in\mathcal{L}$,
and so
$E\not\in\mathcal{L}_P$. 

There is a large family of 
sets of finite perimeter whose quasiconformal images are also
sets of finite perimeter. Indeed,
thanks to the BV co-area formula
(see~\cite[Proposition 4.2]{M})
and the fact that $N^{1,Q}(X)\subset BV_{loc}(X)$, we know that
if $u\in N^{1,Q}(X)$ is compactly supported,
then for $\mathcal{H}^1$-almost every $t\in\R$ we have that
the superlevel set $E_t:=\{u>t\}$ is of finite perimeter in $X$.
Here, by $N^{1,Q}(X)$ we mean the function class $N^{1,Q}(X;\R)$ from Definition~\ref{def:N1p}.
By~\cite[Theorem 9.10]{HKST}
we know that
$u\circ f^{-1}\in N^{1,Q}(Y)$ since $f$ is quasiconformal. 
Therefore for $\mathcal{H}^1$-almost
every $t\in \R$ we have that $f(E_t)=\{u\circ f^{-1}>t\}$ is also of finite perimeter, and so the collection
of all $t\in\R$ for which  either $E_t$ is not of finite perimeter or $f(E_t)$ is not of finite perimeter is
of $\mathcal{H}^1$-measure zero. Hence there are plenty of sets of finite perimeter in $X$ whose
image under $f$ is of finite perimeter in $Y$. The remaining part of this section
is devoted to make concrete the notion of ``plenty", see Remark~\ref{rem:Matti}.
\end{remark}

\begin{proposition}
Let $u\in N^{1,Q}(X)$ be compactly supported
such that $\int_Xg_u^Q\, d\mu>0$,
where $g_u$ is the minimal $Q$-weak upper gradient of $u$ 
({\rm see e.g.~\cite[Section 6]{HKST}}), and let $\mathcal{L}$ be the collection
of all sets $E_t=\{x\in X\, :\, u(x)>t\}$, $t\in\R$,  for which $0<P(E_t,X)<\infty$.
Then $\Mod_{Q/(Q-1)}(\mathcal{L})>0$.
\end{proposition}

\begin{proof}
By employing truncation of $u$ and by adding a constant to $u$ if necessary, we may assume without
loss of generality that $0\le u\le 1$ on $X$ and by the monotonicity of
$\Mod_{Q/(Q-1)}$, we may replace $\mathcal{L}$ with
the collection of all $E_t$ with $0<P(E_t,X)<\infty$,
$0<t<1$. If $P(E_t,X)=0$ for almost every $t\in[0,1]$,
then by the $1$-Poincar\'e inequality we know that $E_t$ is either almost all of $X$ or is of measure zero,
whence we would have $u$ is constant, violating that $\int_Xg_u^Q\, d\mu>0$.
Therefore $\mathcal{L}$ has
many sets, one for each $t$ in a positive $\mathcal{H}^1$-measure subset of $[0,1]$.

Let $\rho$ be admissible for computing $\Mod_{Q/(Q-1)}(\mathcal{L})$.
Then for almost every $t\in [0,1]$,
\[
\int_{\Sigma E_t}\rho \, dP(E_t,\cdot)\ge 1.
\] 
Integrating over $t\in[0,1]$ and applying the co-area formula
and H\"older's inequality, we obtain
\begin{align*}
\mathcal{H}^1(\{t\in[0,1]\, :\, E_t\in \mathcal L\}) &\le \int_0^1 \int_{\Sigma E_t}\rho \, dP(E_t,\cdot)\, dt\\
 & =\int_X\rho\, d\Vert Du\Vert\\
 &\le C\, \int_X \rho\, g_u\, d\mathcal{H}^Q\\
 & \le C\, \left(\int_X\rho^{Q/(Q-1)}\, d\mathcal{H}^Q\right)^{1-\frac{1}{Q}}\, \left(\int_Xg_u^Q\, d\mathcal{H}^Q\right)^{\frac{1}{Q}}.
\end{align*}
Taking the infimum over all such $\rho$ gives
\begin{equation}\label{eq:MakeItObviousPL}
0<\frac{\mathcal{H}^1(\{t\in[0,1] \,:\, E_t\in \mathcal L\})^{Q/(Q-1)}}{C^{Q/(Q-1)}\, 
\left(\int_Xg_u^Q\, d\mathcal{H}^Q\right)^{1/(Q-1)}}
\le \Mod_{Q/(Q-1)}(\mathcal{L}). 
\end{equation}
\end{proof}

Suppose $\int_X g_u^Q\, d\mu>0$ and that $0\le u\le 1$ on $X$. If there is some $0<t_0<1$ for which 
$P(E_{t_0},X)=0$, then either $u>t_0$ almost everywhere on $X$, or else $u\le t_0$ almost everywhere on $X$.
In the former case we have $P(E_t,X)=0$ for all $0<t\le t_0$, and in the latter case we have
$P(E_t,X)=0$ for all $t_0\le t<1$. Thus the set of all $t\in(0,1)$ for which $0<P(E_t,X)<\infty$ is a 
full-measure subset of a subinterval of $[0,1]$.

\begin{remark}
If $t_1<t_2$ and $\int_{\{t_1<u<t_2\}}g_u^Q\, d\mathcal{H}^Q>0$, then for
$\mathcal{L}(t_1,t_2)$, which consists of all $E_t\in\mathcal L$ as in the above proposition for $t_1\le t\le t_2$,
we have 
\[
0<\frac{\mathcal{H}^1(\{t_1<t<t_2 :\, E_t\in \mathcal L\})}{C}
  \le \Mod_{\frac{Q}{Q-1}}(\mathcal{L}(t_1,t_2))^{\frac{Q-1}{Q}}
 \left(\int_{\{t_1<u<t_2\}}g_u^Q\, d\mathcal{H}^Q\right)^{\frac{1}{Q}}.
\]
\end{remark}

\begin{remark}\label{rem:Matti}
As above, we consider the family of sets $E_t$ with $E_t$ superlevel sets for a given compactly
supported function $u\in N^{1,Q}(X)$. Then, with $f^{-1}=:F:Y\to X$ quasiconformal, 
we must have $u\circ F\in N^{1,Q}(Y)$, and
so for almost every $t\in\R$ we have both $E_t$ and $F(E_t)$ are of finite perimeter. Moreover, as noted above,
$P(E_t,X)=0$ if and only if $P(f(E_t),Y)=0$. Furthermore, note that
$f(E_t)=\{y\in Y\, :\, u\circ F(y)>t\}$, and so the
discussion in 
Remark~\ref{rem:MadayanPL} shows that for almost every $t>0$ we have $P(f(E_t), Y)$ is finite. 

Now the proof of
the inequality~\eqref{eq:MakeItObviousPL} tells us that the $Q/(Q-1)$-modulus of the family of all $E_t$ for which
$P(E_t,X)<\infty$ and $P(f(E_t),Y)<\infty$ is positive. Indeed, if $\mathcal{L}_0$ is the collection of all
$E_t$ for which $0<P(E_t,X)<\infty$ but $P(f(E_t),Y)=\infty$, then whenever $\rho_0$ is admissible for computing
$\Mod_{Q/(Q-1)}(\mathcal{L}\setminus \mathcal{L}_0)$, we then have $1\le \int_{\Sigma E_t}\rho \,dP(E_t,\cdot)$
for almost every $t\in [0,1]$ with $E_t\in\mathcal{L}$; thus
the computation that derives~\eqref{eq:MakeItObviousPL} also gives the validity of~\eqref{eq:MakeItObviousPL}
with the role of $\rho$ played by $\rho_0$. That is, $\Mod_{Q/(Q-1)}(\mathcal{L}\setminus \mathcal{L}_0)>0$.
Therefore there are plenty of sets of positive
and finite perimeter
in $X$ whose image under $f$ is also of positive and finite perimeter; that is, 
with $\mathcal{L}_u$ the collection of all $E_t$ for which $0<P(E_t,X)<\infty$ and $P(f(E_t),Y)<\infty$ 
(and hence $0<P(f(E_t),Y)<\infty$)
satisfies
$\Mod_{Q/(Q-1)}(\mathcal{L}_u)>0$ and $\Mod_{Q/(Q-1)}(f(\mathcal{L}_u))>0$ provided $\int_Xg_u^Q\, d\mathcal{H}^Q>0$.
\end{remark}

\noindent Address:

\vskip1ex plus 1ex minus 0.5ex
\noindent R.J. and N.S.: Department of Mathematical Sciences,\\
P. O. Box 210025,
University of Cincinnati,
Cincinnati, OH 45221-0025,
U.S.A.

\vskip1ex plus 1ex minus 0.5ex
\noindent E-mail:
{\tt jones3rh@mail.uc.edu}, {\tt shanmun@uc.edu}
\vskip1ex plus 1ex minus 0.5ex
\noindent P.L. University of Jyvaskyla,
Department of Mathematics and Statistics,\\
P.O. Box 35, FI-40014 University of Jyvaskyla, Finland
\vskip1ex plus 1ex minus 0.5ex
\noindent E-mail: {\tt panu.k.lahti@jyu.fi}

\end{document}